\documentclass[reqno,11pt]{amsart}

\usepackage{amsmath,amssymb,amsfonts,amsthm, amscd,indentfirst}
\usepackage{amsmath,latexsym,amssymb,amsmath,
	amscd,amsthm,amsxtra,mathrsfs}
 
\usepackage[top=3cm, bottom=3.5cm, right=2.5cm, left=2.2cm]{geometry}

\usepackage[hyphens]{url} 
\usepackage{hyperref}
\usepackage{bookmark}
\usepackage{amsmath,thmtools,mathtools}
\usepackage{etoolbox}

\allowdisplaybreaks 
\mathtoolsset{showonlyrefs=true}
\DeclareRobustCommand{\textsection}{\ifmmode\mathsection\else\S\fi}

\usepackage[
  backend=biber,
  style=alphabetic,
  sorting=nyt,
  maxbibnames=99,
  giveninits=true
]{biblatex}
\DeclareFieldFormat[article,inproceedings,incollection,unpublished]{title}{#1}
\DeclareFieldFormat[article,inproceedings,incollection,unpublished]{citetitle}{#1}

\newcounter{mparcnt}

\usepackage{fancyhdr}
\usepackage{esint}
\usepackage{enumerate}
\usepackage{xcolor}

\usepackage{pictexwd,dcpic}
\usepackage{graphicx}
\usepackage{caption}

\usepackage{graphicx}
\usepackage{caption}
\usepackage{slashed}

\usepackage{epsfig,here}
\usepackage{subfigure,here}

\newtheorem{theorem}{Theorem}[section]
\newtheorem{lemma}[theorem]{Lemma}
\newtheorem{proposition}[theorem]{Proposition}
\newtheorem{definition}[theorem]{Definition}
\newtheorem{corollary}[theorem]{Corollary}
\newtheorem{remark}[theorem]{Remark}

\newcommand{\abs}[1]{\lvert#1\rvert}
\newcommand{\Abs}[1]{\left\lvert#1\right\rvert}

\newcommand{\norm}[1]{\lVert#1\rVert}

\newcommand{\rd}{{\rm d}}
\newcommand{\rdV}{{\rm dV}}

\newcommand{\rVol}{{\rm Vol}}

\newcommand{\rtr}{{\rm tr}}

\newcommand{\D}{{\slashed{D}}}

\newcommand{\curl}{{\rm curl}\,}
\newcommand{\ip}{\lrcorner\,}
\newcommand{\w}{\wedge}

\def\<{\langle}
\def\>{\rangle}
\def\S{\mathbb{S}}
\def\R{\mathbb{R}}

\newcommand{\ra}{\rightarrow}

\def\SS{{\mathbb S}}

\newcommand{\eq}[1]{\begin{equation}\allowdisplaybreaks\begin{alignedat}{2} #1 \end{alignedat}\end{equation}}

\numberwithin{equation} {section}

\begin{document}

	
\title[On the sharp constants in curl-Sobolev inequalities on $\mathbb{S}^n$]{On the sharp constants in curl-Sobolev inequalities on $\mathbb{S}^n$}
\date{\today}


\author{Guofang Wang}\address{ Albert-Ludwigs-Universit\"at Freiburg,
Mathematisches Institut,
Ernst-Zermelo-Str. 1,
D-79104 Freiburg, Germany}
\email{guofang.wang@math.uni-freiburg.de}

\author{Mingwei Zhang}\address{ Wuhan University, School of Mathematics and Statistics, 430072 Wuhan, China and 
Albert-Ludwigs-Universit\"at Freiburg,
Mathematisches Institut,
Ernst-Zermelo-Str. 1,
D-79104 Freiburg, Germany}
\email{zhangmwmath@whu.edu.cn}

\begin{abstract}
Let $n\equiv 3\ (\mathrm{mod}\ 4)$ and set $p=\frac{n-1}{2}$. On an oriented Riemannian $n$-manifold we consider the (middle-degree) curl operator, $\curl\coloneqq*\rd:\Omega^{p}\to\Omega^{p}$, and the associated conformally invariant Sobolev quotients on $(\S^n,g_{\mathrm{st}})$,
\[
J_1(\alpha)=\frac{\big(\int|\curl\alpha|^{\frac{2n}{n+1}}\,\rdV\big)^{\frac{n+1}{n}}}{\int\langle\curl\alpha,\alpha\rangle\,\rdV},
\qquad
J_2(\alpha)=\frac{\big(\int|\curl\alpha|^{\frac{2n}{n+1}}\,\rdV\big)^{\frac{n+1}{n}}}{\inf_{\phi}\big(\int|\alpha-\rd\phi|^{\frac{2n}{n-1}}\,\rdV\big)^{\frac{n-1}{n}}}.
\]
Killing $p$-forms and their conformal images form a natural family of critical points for both functionals, analogous to the Aubin--Talenti family in the classical Sobolev inequality. We prove a quantitative local stability estimate for $J_1$ around this family, which in particular implies that every such form is a strict local minimizer in the conformally invariant space $W^{1,\frac{2n}{n+1}}$. In contrast, we show that these critical points are unstable for $J_2$ (and for related conformally invariant quotients), yielding a strict upper bound for the sharp constant of the $J_2$ inequality. By conformal invariance, the results on $\S^n$ transfer naturally to $\R^n$.


\end{abstract}

\maketitle
\tableofcontents

\section{Introduction}\label{sec1}

Let $M$ be an oriented Riemannian manifold of dimension $n\equiv 3\!\mod 4$, and write $\Omega^p$ for the space of smooth $p$-forms on $M$. In the middle degree $p=\frac{n-1}{2}$ we consider the \emph{curl operator}
\eq{
    \curl \coloneqq *\rd: \Omega^{\frac{n-1}{2}}\to \Omega^{\frac{n-1}{2}},
}
where $\rd$ is the exterior derivative and $*$ the Hodge star. The middle-degree restriction is forced by grading: $\rd$ raises degree by one and $*$ identifies $\Omega^{p+1}$ with $\Omega^{n-p-1}$, hence $*\rd$ preserves degree exactly when $p=\frac{n-1}{2}$. 
For $n\equiv 3\!\mod 4$, $\curl$ is real and self-adjoint on $L^2(\Omega^{\frac{n-1}{2}})$. It is, however, non-elliptic: $\ker(\curl\!)$ contains the infinite-dimensional space $\rd\Omega^{\frac{n-3}{2}}$ of exact forms. Despite this, the nonzero spectrum is discrete with finite multiplicities. Moreover, since $\rd^*=\pm *\rd*$,
\eq{
    \curl^2\alpha = (*\rd)(*\rd)\alpha = \pm\,\rd^*\rd\,\alpha,
}
so on co-closed forms ($\rd^*\!\alpha=0$) the square of $\curl$ agrees with the Hodge Laplacian $\Delta=\rd\rd^*+\rd^*\rd$.

In dimension $3$ (so $p=1$), identifying vector fields with dual $1$-forms via the metric gives $(*\rd)(u^\flat)=(\nabla\times u)^\flat$.
This framework ties $\curl$ to classical models in mathematical physics. In Maxwell theory, $\nabla\times E=-\partial_t B$ and $\nabla\times H=J+\partial_t D$, so curl estimates control electromagnetic fields under natural constraints. In fluid dynamics and magnetohydrodynamics, the vorticity and current are $\nabla\times u$ and $\nabla\times B$; eigenfields $\nabla\times u=\lambda u$ (Beltrami fields) describe force-free magnetic fields and arise as steady incompressible Euler flows and in energy--helicity variational principles. Further motivations appear in the final paragraphs of this Introduction.

Since the variational problems considered below are conformally invariant (see Section~\ref{sec2.2}), it is enough to work on the standard sphere $(\S^n,g_{\rm st})$: via stereographic projection, our results may be transported to $\R^n$. We therefore focus on $\S^n$ and study two conformally invariant Sobolev quotients for $\frac{n-1}{2}$-forms:
\eq{
    J_1(\alpha) \coloneqq \frac{ \Big(\int \abs{\curl\alpha}^{\frac{2n}{n+1}} \,\rdV \Big)^{\frac{n+1}{n}} }{ \int \< \curl\alpha, \alpha\> \,\rdV},
}
\eq{
    J_2(\alpha) \coloneqq \frac{\Big(\int \abs{\curl \alpha}^{\frac{2n}{n+1}} \,\rdV\Big)^{\frac{n+1}{n}}}{\inf_{\phi\in\Omega^{\frac{n-3}{2}}} \Big(\int \abs{\alpha - \rd\phi}^{\frac{2n}{n-1}} \,\rdV\Big)^{\frac{n-1}{n}}},
}
together with conformally invariant quotients $J_{p,k}$ defined below.
Here the integration is over $\S^n$ and $\rdV$ is its standard volume element.
The corresponding sharp constants are
\eq{
    S_1 \coloneqq \inf_{\int \< \curl\alpha, \alpha\> >0}J_1(\alpha), \qquad S_2 \coloneqq \inf_{\alpha\not\equiv 0}J_2(\alpha).
}
By the Hodge decomposition (see Section~\ref{sec2}), we may, without loss of generality, restrict attention to the subspace $\{\alpha\mid \rd^*\!\alpha=0\}$.
It is easy to check that $S_1,S_2>0$.
By contrast, proving the existence and characterizing the extremals and determining the sharp values of $S_1$ and $S_2$ are more delicate. For $J_1$
the existence of minimizers was proved already by Rivi\`ere in \cite{Riviere98CAG}, in an interesting connection with $p$-harmonic maps.
For $J_2$ in the case $n=3$, the existence of minimizers was proved by Frank and Loss~\cite{FL22}; their approach extends to all $n\ge 3$. See also the related work~\cite{MS25}. The paper addresses the problem of the sharp constants $S_1$ and $S_2$ and their possible extremals. 

A natural first step toward the sharp constants is to identify distinguished critical points. The Euler--Lagrange equation for $J_1$ is
\eq{\label{eq_Y1}
    \curl (\abs{\curl \alpha}^{-\frac 2 {n+1}} \curl \alpha) =\mu_1 \, \curl \alpha \quad\hbox{for some constant}\ \mu_1>0,
}
or equivalently, setting $\beta \coloneqq \abs{\curl \alpha}^{-\frac{2}{n+1}} \curl \alpha$, we obtain
\eq{ \label{eq_Y1'}
    \curl \beta =\mu_1 \abs{\beta}^{\frac 2{n-1}} \beta.
}
The Euler--Lagrange equation  of $J_2$ can be written as
\eq{\label{eq_Y2}
    \curl (|\curl \alpha |^{-\frac 2 {n+1}} \curl \alpha) =\mu_2 \,\abs{ \alpha - \rd\phi}^{\frac{2} {n-1}} (\alpha - \rd\phi)  \quad\hbox{for some constant}\ \mu_2>0,
}
and, after replacing $\alpha-\rd\phi$ by $\alpha$, equivalently as
\eq{\label{eq_Y2'}
\curl (|\curl \alpha |^{-\frac 2 {n+1}} \curl \alpha) =\mu_2 \,\abs{ \alpha }^{\frac{2} {n-1}} \alpha.
}
In particular, any solution of \eqref{eq_Y1'} is also a solution of \eqref{eq_Y2'}.

A distinguished family of solutions is provided by \emph{Killing forms}; see Definition~\ref{def_Killing} and Proposition~\ref{Killing_eigenform}. On $\S^n$ these are precisely the first (nonzero) curl eigenforms:
\eq{
\curl \xi=\frac {n+1}2 \xi,
}
and they have constant length: $\abs{\xi}\equiv\mathrm{const}$. Consequently, every Killing form solves both \eqref{eq_Y1} and \eqref{eq_Y2'}. 
A key feature of the functionals $J_1$ and $J_2$ is their conformal invariance.
Hence, 
their conformal images, i.e., deformations under conformal transformations, are also solutions. We denote by
\eq{\mathcal{M} :=\{\hbox{Killing forms and their conformal transformations}\}.}
The family $\mathcal M$ is the natural analogue, for the curl operator, of the Aubin--Talenti family for the classical Sobolev inequality on $\S^n$. This leads to the basic question: are elements of $\mathcal M$ optimizers for $J_1$ and/or $J_2$? A positive answer would yield the explicit candidates
$S_1 =\frac {n+1} 2 \omega_n^{1\slash n}$ and $S_2 = \frac {(n+1)^2 }4 \omega_n^{2 \slash n}$.

\medskip
The first part of this paper shows that $\mathcal M$ is locally optimal for $J_1$ in the conformally invariant topology, while the second part shows that  elements in $\mathcal M$ are not  minimizers of $J_2$, in fact, they are unstable. The latter result remains true for other Sobolev-type functionals. See below.

\begin{theorem}\label{thm1.1}
Let $n\equiv 3\! \mod 4$. There exist constants $\delta_0>0$ and $c(n)>0$  such that for any $\alpha\in W^{1, \frac{ 2n}{n+1}}$ with 
with $\int \<\curl \alpha, \alpha \> >0$ and 
\eq{
     \frac{\inf_{\beta\in\mathcal{M}}\Big(\int\abs{\curl(\alpha-\beta)}^{\frac{2n}{n+1}}\Big)^{\frac{n+1}{n}}}{\int\<\curl\alpha,\alpha\>} <\delta_0,
}
we have 
\eq{
    \frac{\Big(\int \Abs{\curl\alpha}^{\frac{2n}{n+1}}\Big)^{\frac{n+1}{n}}}{\int\<\curl\alpha,\alpha\>}
    - \frac{n+1}{2}\omega_{n}^{\frac{1}{n}}
    \geq c(n)\,\frac{\inf_{\beta\in\mathcal{M}}\Big(\int\abs{\curl(\alpha-\beta)}^{\frac{2n}{n+1}}\Big)^{\frac{n+1}{n}}}{\int\<\curl\alpha,\alpha\>}.
}
\end{theorem}

We call $\xi\in\mathcal M$ a \emph{strict local minimizer} of $J_1$ (in $W^{1,\frac{2n}{n+1}}$) if $J_1(\alpha)\ge J_1(\xi)$ for all $\alpha$ in a sufficiently small neighborhood of $\xi$, with equality only when $\alpha\in\mathcal M$. Theorem~\ref{thm1.1} yields this local minimality, modulo the conformal symmetries.

\medskip
In contrast, for $J_2$ we obtain an instability result. More generally, following Frank--Loss~\cites{FL22, FL1}, for $k$-forms we consider
\eq{
    J_{p,k}(\alpha) \coloneqq \frac {\big( \int \abs{\curl \alpha}^{p} \,\rdV\big)^{q/p}} {\int \abs{\alpha}^q \,\rdV},\quad\alpha\in\Omega^k, \quad \hbox{with}\quad 1<p<n \quad \hbox{and}\quad q=\frac{np}{n-p}.
}
(See Section~\ref{sec6}.) $J_{p,k}$ is conformally invariant if and only if $p=\frac{n}{k+1}$. In particuar, when $k=\frac{n-1}{2}$ and $p=\frac{2n}{n+1}$ one has $J_{\frac{2n}{n+1},\frac{n-1}{2}}=J_2$.

\begin{theorem}\label{thm1.2}
Each $\xi\in\mathcal{M}$ is not a local minimizer of $J_2=J_{\frac{2n}{n+1},\frac{n-1}{2}}$ in $W^{1, \frac{2n}{n+1}}$.
Moreover, each Killing $\frac{n-1}{2}$-form is not a local minimizer of $J_{p,\frac{n-1}{2}}$ in $W^{1,p}$ for $1<p<n$.
\end{theorem}

\begin{theorem}\label{thm1.3}
Each $\xi\in\mathcal{M}$ is not a local minimizer of $J_{\frac{n}{2},1}$ in $W^{1, \frac n 2}$. \end{theorem}

In fact, in the proof for both Theorems~\ref{thm1.2} and~\ref{thm1.3} we show that $\xi\in\mathcal{M}$ is an unstable critical point; it would be interesting to determine its Morse index. We believe that each 
$\xi \in \mathcal M$ is also not a local minimizer for $J_{\frac n{k*1}, k},$  when  $2\le k\le n-1$. 

\medskip
The case $n=3$ is of particular importance. It was asked in \cite{Frank_Loss_2024} and more recently in \cite{MS25} whether
\eq{\label{conjecture_S2}
    S_2=\frac {(n+1)^2 }4 \omega_n^{\frac{2}{n}}.
}
Theorem~\ref{thm1.2} gives a negative answer.

For $J_1$ in dimension $3$, it was proved in \cite{EGP25} that Killing forms are local minimizers in the $C^0$ topology. Our result differs in two essential ways: we work in the conformally invariant space $W^{1,\frac{2n}{n+1}}$, and we obtain local stability in a neighborhood of the full conformal family $\mathcal M$ (which is larger than the set of Killing forms). In particular, due to conformal invariance, our results hold also in $\R^n$. We also note that \cite{EGP25} is phrased in terms of an equivalent problem on optimal metrics for the first curl eigenvalue; see also Appendix~\ref{secB}. The computations given in \cite{EGP25} are very difficult to be generalized to higher dimensions. We recently realized that it was already proved by Rivi\`ere in \cite{Riviere98CAG} that the Killing forms are local optimizers in $C^1$.

At present we cannot prove or disprove that elements of $\mathcal M$ are \emph{global} minimizers of $J_1$; this remains a challenging open problem. If true, Theorem~\ref{thm1.1} would yield a quantitative stability inequality in the spirit of \cite{WZ25}.\footnote{After the completion of the present manuscript, we succeeded to solve this problem in a subsequent preprint, namely we confirmed that elements of $\mathcal M$ are \emph{global} minimizers of $J_1$, hence the quantitative stability follows.}

Finally, it is useful to compare $\frac{n-1}{2}$-forms with spinor fields: Killing $\frac{n-1}{2}$-forms play a role analogous to Killing spinors, and the curl operator shares several features with the Dirac operator. Our approach is motivated by methods developed for spinors in \cite{WZ25}.

Before ending the introduction, we give  three more concrete geometric 
motivations for the study of  optimal Sobolev inequalities for differential forms.

\medskip
\noindent{\it Motivation from zero modes.} In dimension $3$, a magnetic potential can be modeled by a vector field $A$ and its magnetic field by $\rd A^\flat$ (equivalently $\nabla\times A$). Gauge invariance corresponds to the freedom to replace $A$ by $A+\nabla\phi$, which leaves $\rd A^\flat$ unchanged. The existence of zero modes for Pauli/Dirac operators in magnetic fields, beginning with the construction of Loss--Yau \cite{Loss_Yau_86}, is intimately connected to borderline (critical) estimates that control a potential $A$ modulo gauge in terms of its curl.

More precisely, in the critical regime one seeks sharp, conformally invariant inequalities of the schematic form
\eq{
    \inf_{\phi}\,\|A-\nabla\phi\|_{L^{\frac{2n}{n-1}}}\ \leq\ C\,\|\nabla\times A\|_{L^{\frac{2n}{n+1}}},
}
whose best constant controls how small a gauge-fixed potential can be relative to its magnetic field. Frank--Loss \cite{Frank_Loss_2024} made this link explicit and related the possibility of zero modes to sharp curl inequalities after optimizing over exact perturbations. The functional $J_2$ is tailored to this structure: its denominator precisely performs the gauge optimization $\alpha\mapsto \alpha-\rd\phi$, and the corresponding sharp constant encodes the optimal critical control of the potential by its curl.

\medskip

\noindent{\it Motivation from the first curl eigenvalue.} In dimension $3$, identifying vector fields with $1$-forms, the restriction of $\curl$ to divergence-free fields (equivalently, co-closed $1$-forms) is self-adjoint on $L^2$. Its first nonzero eigenvalue $\lambda_1$ is characterized variationally by a Rayleigh quotient and governs the ``simplest'' nontrivial stationary solutions of Euler: eigenfields
\eq{
    \nabla\times u = \lambda u
}
are Beltrami fields, i.e., critical points of the energy under a helicity constraint (and, in suitable settings, candidates for stable steady flows). This makes $\lambda_1$ a natural spectral quantity in hydrodynamic stability and in variational principles in magnetohydrodynamics.

From a geometric viewpoint one can then ask how $\lambda_1$ depends on the underlying geometry: for instance, how it changes under deformation of the metric, or under changes of the domain subject to volume or convexity constraints. Such ``optimal eigenvalue'' problems for curl have been investigated in depth, e.g. on optimal metrics and optimal convex domains in \cites{EGP25,EGP22}, as well as in related nonexistence results \cite{EP20}. See also Appendix B. The quotients $J_1$ and $J_2$ may be viewed as conformally invariant, form-valued analogues of these spectral problems: rather than fixing volume and minimizing a quadratic Rayleigh quotient, one works at the critical Sobolev scaling, which is the natural scale for conformal geometry.

\medskip
\noindent{\it Motivation from $p$-harmonic maps.} Curl--Sobolev-type inequalities enter naturally in the variational theory of $p$-harmonic maps $u:\S^3\to\S^2$. Following Rivi\`ere \cite{Riviere98CAG}, one encodes the Jacobian-type nonlinearity of $u$ by the closed $2$-form
\eq{
    u^*\omega_{\S^2} = \rd\alpha,
}
where $\omega_{\S^2}$ is the area form on $\S^2$, and $\alpha$ is the potential $1$-form. The elementary pointwise estimate
\eq{\label{eq_pointwise_F_du}
    \abs{\rd\alpha}=\abs{u^*\omega_{\S^2}}\ \leq\ \frac12\,\abs{\rd u}^2
}
implies, for any $q>1$, the integral bound $\int_{\S^3}\abs{\rd\alpha}^q\le 2^{-q}\int_{\S^3}\abs{\rd u}^{2q}$ (cf. \cite[(1.11)]{Riviere98CAG}). Equality in \eqref{eq_pointwise_F_du} characterizes transversally conformal maps and is attained by the Hopf fibration.

It is well known that when $n=3$, the denominator $\int_{\S^3}\<\curl\alpha,\alpha\>$ of $J_1$ is a homotopy invariant (the Hopf invariant). Therefore, to minimize the $p$-energy of $u$ in a given homotopy class, one can instead minimize the functional $J_1$ of $\alpha$, hence related to the optimal Sobolev inequality for differential forms.

Rivi\`ere further introduced a conformally invariant constrained variational problem for the potentials (imposing a Coulomb gauge to obtain coercivity). In that framework, the first eigenforms of the linear operator $\Delta^{1/2}$ (equivalently, of $\curl$ acting on co-closed $1$-forms) arise as stable critical points of the associated nonlinear functional. In particular, one can deduce that the Hopf fibration minimizes the $p$-energy in its homotopy class for $p\ge 4$, and remains locally minimizing for $3\le p<4$ \cite{Riviere98CAG}. This provides a geometric reason why sharp, gauge-optimized curl inequalities and their extremals are natural tools in the fine analysis of $p$-harmonic maps. In the same vein, our variational quantities are also closely related to the Faddeev--Skyrme model on $\S^3$, whose energy involves both the Dirichlet term $|\rd u|^2$ and the ``Skyrme term'' $|u^*\omega_{\S^2}|^2$ built from the same $2$-form as above.

\medskip

\noindent{\it Organization of the rest of the paper.} Section~2 collects preliminaries on exterior calculus, the curl operator and Killing forms. In Section~3 we prove sharp estimates for curl eigenforms, which are the key input for the stability analysis. Section~4 proves Theorem~\ref{thm1.1}. Section~5 proves the conjecture \eqref{conjecture_S2} is not true. Section~6 discusses the generalized quotients and proves Theorems~\ref{thm1.2} and~\ref{thm1.3}. Appendix~\ref{secA} contains a classification of curl eigenforms, and Appendix~\ref{secB} introduces a related conformal invariant and the problem of optimal metrics. 
In Appendix \ref{secC} we give a quantitative lower bound of $S_1$ ($n=3$) by using the optimal Sobolev type inequality for spinors.

\section{Preliminaries}\label{sec2}

\subsection{Exterior calculus}

Let $(M,g)$ be an oriented Riemannian manifold of dimension $n\equiv 3\!\mod 4$. For each $0\le p\le n$, we write $\Omega^p$ for the space of smooth $p$-forms and endow it with the $L^2$ inner product induced by $g$. The exterior differential is denoted by $\rd: \Omega^p\to \Omega^{p+1}$, where by convention $\Omega^{n+1}=0$. Fix the orientation determined by the volume form $\rdV$. The Hodge star operator $*:\Omega^p\to \Omega^{n-p}$ is characterized by
\eq{
    \alpha\w*\beta = \<\alpha, \beta\>\,\rdV.
}
The codifferential (the $L^2$-adjoint of $\rd$) is
\eq{
    \rd^*\coloneqq (-1)^{n(p+1)+1}*\rd*:\Omega^p\ra\Omega^{p-1}.
}
Since $n\equiv 3\!\mod 4$, we have $\rd^*=-*\rd*$ on $\Omega^{\frac{n-1}{2}}$. We recall the standard identities
\eq{
    *^2 = (-1)^{p(n-p)}\,{\rm id},\qquad \rd^2 = (\rd^*)^2=0.
}
In particular, throughout this paper $*^2 = {\rm id}$.

The Hodge Laplacian is defined by
\eq{
    \Delta \coloneqq (\rd+\rd^*)^2 = \rd\rd^* + \rd^*\rd: \Omega^p\ra\Omega^p
}
and the Bochner (or rough) Laplacian by
\eq{
    \nabla^*\nabla \coloneqq -\rtr\nabla^2: \Omega^p\ra\Omega^p.
}
They are related by the Weitzenb\"ock formula
\eq{
    \Delta = \nabla^*\nabla + \sum_{i,j} e_j^\flat \w e_i \ip R(e_i, e_j),
}
where $\{e_i\}$ is a local orthonormal frame, $e_i^\flat$ are the dual $1$-forms, $\ip$ denotes interior multiplication, and $R$ is the curvature tensor $R(X,Y)=\nabla_X\nabla_Y-\nabla_Y\nabla_X-\nabla_{[X,Y]}$. On the standard sphere $\SS^n$, acting on $p$-forms, this reduces to
\eq{\label{sphere_Weitzenbock}
    \Delta = \nabla^*\nabla + p(n-p).
}

If $M$ is closed, the Hodge decomposition gives
\eq{
    \Omega^p = {\rm ker}(\Delta)\oplus \rd\Omega^{p-1}\oplus \rd^*\Omega^{p+1},
}
and hence
\eq{
    {\rm ker}(\rd) = {\rm ker}(\Delta)\oplus \rd\Omega^{p-1},\qquad {\rm ker}(\rd^*) = {\rm ker}(\Delta)\oplus \rd^*\Omega^{p+1}.
}
Since $H^p(\S^n)=0$ for $0<p<n$, we obtain in particular on $\S^n$
\eq{\label{sphere_Hodge}
    \Omega^{\frac{n-1}{2}} = \rd\Omega^{\frac{n-3}{2}}\oplus \rd^*\Omega^{\frac{n+1}{2}},\quad {\rm ker}(\rd) = \rd\Omega^{\frac{n-3}{2}},\quad {\rm ker}(\rd^*) = \rd^*\Omega^{\frac{n+1}{2}}.
}
We next record some basic identities that will be used repeatedly.

\begin{proposition}\label{basic_form}
Let $(M,g)$ be an oriented closed Riemannian manifold of dimension $n\equiv 3\!\mod 4$. Then for any $\alpha,\beta\in\Omega^{\frac{n-1}{2}}$ and $X\in\Gamma(TM)$ we have
\begin{enumerate}
    \item $\abs{X\ip\alpha}^2+\abs{X^\flat\w\alpha}^2=\abs{X}^2\abs{\alpha}^2$ (Pythagoras theorem);
    \item $X\ip*\!\alpha=-*(X^\flat\w\alpha)$ (interior product dual to exterior product);
    \item $\<X\ip*\!\alpha,\beta\>=-\<\alpha, X\ip*\beta\>$ (anti-symmetric), and as a consequence, $\<X\ip*\!\alpha,\alpha\>=0$;
    \item $\rd = \sum_i e_i^\flat\w\nabla_{e_i}$,\quad $\rd^*=-\sum_i e_i\ip\nabla_{e_i}$ (local expression);
    \item $\mathcal{L}_X\alpha = \rd (X\ip\alpha) + X\ip\rd\alpha$ (Cartan’s magic formula);
    \item $\sum_i e_i^\flat\w e_i\ip\alpha = \frac{n-1}{2}\alpha$,\quad $\sum_i e_i\ip e_i^\flat\w\alpha = \frac{n+1}{2}\alpha$ (counting formula);
    \item $e_i^\flat\w e_j\ip\alpha + e_j\ip e_i^\flat\w\alpha = \delta_{ij}\alpha$ (Clifford multiplication).
  \end{enumerate}
More generally, for $\alpha,\beta\in\Omega^p$ the above formulas remain valid, with the following modifications:
\begin{enumerate}
    \item[(2)] $X\ip*\!\alpha=(-1)^p*(X^\flat\w\alpha)$; 
    \item[(3)] $\<X\ip*\!\alpha,\beta\>=(-1)^p\<\alpha, X\ip*\beta\>$; 
    \item[(6)] $\sum_i e_i^\flat\w e_i\ip\alpha = p\alpha$, \quad $\sum_i e_i\ip e_i^\flat\w\alpha = (n-p)\alpha$.
\end{enumerate}
As a direct consequence, we have
\eq{\label{direct_dual}
    \<X\ip\alpha,\beta\> = \<\alpha,X^\flat\w\beta\>, \quad\forall \alpha\in\Omega^p,\ \beta\in\Omega^{p-1}.
}
\end{proposition}

\subsection{The curl operator}\label{sec2.2}

In this subsection we discuss the conformal covariance of the curl operator and recall its spectrum on $\S^n$. For $p=\frac{n-1}{2}$ we define
\eq{
    \curl = *\rd: \Omega^{\frac{n-1}{2}}\ra \Omega^{\frac{n-1}{2}}.
}
Let $\tilde g=u^2g$ for some smooth function $u>0$. The standard conformal transformation laws are
\eq{\label{conformal_change}
    *_{\tilde{g}} = u^{n-2p} *_{g}, \quad \<\cdot,\cdot\>_{\tilde{g}} = u^{-2p}\<\cdot,\cdot\>_g \quad\hbox{on}\ \Omega^p.
}
In particular, for any $\alpha\in\Omega^{\frac{n-1}{2}}$,
\eq{
    \curl_{\tilde{g}}\alpha = *_{\tilde{g}}\rd\alpha = u^{-1}\,\curl_g\alpha.
}
Together with \eqref{conformal_change} this yields
\eq{
    \abs{\curl_{\tilde{g}}\alpha}_{\tilde{g}}^{\frac{2n}{n+1}} = u^{-n}\abs{\curl_g\alpha}_g^{\frac{2n}{n+1}},\quad
    \<\curl_{\tilde{g}}\alpha, \alpha\>_{\tilde{g}} = u^{-n}\<\curl_g\alpha,\alpha\>_g,\quad
    \abs{\alpha}_{\tilde{g}}^{\frac{2n}{n-1}} = u^{-n}\abs{\alpha}_g^{\frac{2n}{n-1}}.
}
Consequently, the three quantities
\eq{\label{functionals}
    \int \abs{\curl\alpha}^{\frac{2n}{n+1}}\,\rdV, \quad \int \<\curl\alpha,\alpha\>\,\rdV, \quad \int \abs{\alpha}^{\frac{2n}{n-1}}\,\rdV
}
are conformal invariants, and hence so are $J_1$ and $J_2$.

The space ${\rm ker}(\rd)=\{\alpha\mid \rd\alpha=0\}$ is independent of the metric, whereas its $L^2$-orthogonal complement depends on the metric. By the Hodge decomposition,
\eq{
    {\rm ker}(\rd)^{\perp} = {\rm ker}(\rd^*)
}
(with $\perp$ taken in the $L^2$ inner product induced by $g$). Thus, for conformally invariant expressions depending only on $\curl\alpha=*\rd\alpha$, it is natural to work modulo closed forms, i.e.
\eq{
    \Omega^{\frac {n-1}2}\slash \,{\rm ker} (\rd).
}

The curl operator has purely point spectrum: the eigenvalue $0$ has infinite multiplicity, while each nonzero eigenvalue has finite multiplicity. If $(M,g)$ admits an orientation-reversing isometry (for instance, if $M$ is a symmetric space), then the spectrum is symmetric about $0$. Excluding the eigenvalue $0$, all eigenforms are co-closed. In fact, for any $\lambda\neq0$, by definition of the Hodge Laplacian we have (see \cite{B19})
\eq{\label{combination}
    \alpha = \alpha_+ + \alpha_- \quad \hbox{ with }\quad \curl\alpha_+=\lambda\alpha_+ \quad \hbox{and}\quad \curl\alpha_-=-\lambda\alpha_-
}
if and only if
\eq{
\Delta\alpha = \lambda^2\alpha \quad \hbox{and}\quad \alpha\in \rd^*\Omega^{\frac{n+1}{2}}.
}
In particular, the spectrum of curl on $\S^n$ can be determined by a result in \cite{IK79}. 
\begin{proposition}[cf. \cites{IK79,B19}]\label{eigenvalue_and_multi}
The spectrum of the curl operator on $(\S^n, g_{{\rm st}})$, besides $0$, is
\eq{
    \left\{ \pm\left(\frac{n+1}{2}+k\right) \,\Big|\, k\geq 0 \right\},
}
and the corresponding multiplicity is
\eq{
    m(\curl, \lambda) = \frac{ (n+k)! }{ \left(\frac{n-1}{2}!\right)^2\cdot k!\cdot \left(\frac{n+1}{2}+k\right) }.
}
\end{proposition}

\subsection{Killing forms}

We next recall Killing forms and fix notation for the relevant eigenspaces.

\begin{definition}\label{def_Killing}
A co-closed $p$-form $\alpha\in\Omega^p$ is called a \emph{Killing $p$-form} if
\eq{
    \nabla_X\alpha = \frac{1}{p+1}X\ip\rd\alpha,\quad \forall X\in\Gamma(TM).
}
\end{definition}

For background on Killing and conformal Killing forms we refer to \cites{Semmelmann03, Stromenger10}. On the round sphere one has the following characterization.

\begin{proposition}
\label{Killing_eigenform}
A co-closed $\frac{n-1}{2}$-form $\xi$ on $(\S^n, g_{{\rm st}})$ is a Killing form if and only if it is a first (co-closed) eigenform of the Hodge Laplacian, namely
$\Delta\xi=\frac{(n+1)^2}{4}\xi$,
and it admits a decomposition as in \eqref{combination}.
We call $\xi$ a positive/negative Killing form if $\curl\xi = \pm\frac{n+1}{2}\xi$.
\end{proposition}

Throughout the paper $\xi$ denotes a Killing $\frac{n-1}{2}$-form on $\S^n$.
For $k\ge 1$ we write $E_{\pm k}$ for the curl eigenspace with eigenvalue $\pm\big(\frac{n+1}{2}+k-1\big)$.
We also denote by $P_k$ ($k\ge 0$) the space of spherical harmonics of degree $k$, i.e. eigenfunctions of $\nabla^*\nabla$ with eigenvalue $k(n+k-1)$.
By Proposition~\ref{Killing_eigenform}, $E_1$ (resp. $E_{-1}$) is precisely the space of positive (resp. negative) Killing $\frac{n-1}{2}$-forms.

We now record several properties of Killing forms that will be used repeatedly.

\begin{proposition}\label{Killing_properties}
Let $\xi\in E_1$ and $\eta\in E_{-1}$. Then
\begin{enumerate}
    \item $\nabla_X\xi = X\ip*\xi$ and $\nabla_X\eta = -X\ip*\eta$ for any vector $X$;
    \item $\xi$ and $\eta$ have constant length;
    \item Any orthonormal basis of $E_1$ trivializes the $\frac{n-1}{2}$-form bundle; the same holds true for $E_{-1}$;
    \item For any function $f$, we have
    \eq{
    \curl (f\xi) = \frac{n+1}{2}f\xi - \nabla f\ip *\xi;
    }
    \item 
    If $f\in P_1$, then
    \eq{
    \curl (\nabla f\ip *\xi) = -\frac{n+1}{2}f\xi + \nabla f\ip *\xi=-\curl (f\xi).
    }
\end{enumerate}
\end{proposition}

\begin{proof}
(1) By Proposition~\ref{Killing_eigenform} we have $\curl\xi = \frac{n+1}{2}\xi$. Combining this with Definition~\ref{def_Killing}, we compute
\eq{
    \nabla_X\xi = \frac{2}{n+1}X\ip \rd\xi = \frac{2}{n+1}X\ip *\curl\xi = X\ip *\xi.
}
The statement for $\eta\in E_{-1}$ is proved similarly.

(2) This follows immediately from (1) and Proposition~\ref{basic_form}(3).

(3) Let $\xi_1,\xi_2\in E_1$ and $X\in\Gamma(TM)$. Using (1) and Proposition~\ref{basic_form}(3), we obtain
\eq{
    \nabla_X\<\xi_1,\xi_2\> = \<\nabla_X\xi_1, \xi_2\> + \<\xi_1, \nabla_X\xi_2\>
    = \<X\ip*\xi_1, \xi_2\> + \<\xi_1, X\ip*\xi_2\> = 0.
}
Hence $\<\xi_1,\xi_2\>$ is constant. Moreover, by Proposition~\ref{eigenvalue_and_multi} we have
\eq{
    {\rm dim}E_1 = m(\curl, \frac{n+1}{2}) = \binom{n}{\frac{n-1}{2}} = {\rm rank}\,\Omega^{\frac{n-1}{2}}.
}
Therefore any orthonormal basis $\{\xi_i\}$ of $E_1$ gives a global frame of $\Omega^{\frac{n-1}{2}}$, which is parallel with respect to the connection
\eq{
    \widetilde{\nabla}_X=\nabla_X - X\ip*.
}
Using Proposition~\ref{basic_form}(2) and
\eq{
    R^{\nabla}(X,Y) = Y^\flat \wedge X\ip - X^\flat \wedge Y\ip \quad \hbox{on}\ \S^n,
}
one checks that
\eq{
    R^{\tilde{\nabla}}=0.
}
The corresponding statements for $E_{-1}$ are analogous.

(4) Using Proposition~\ref{basic_form}(2) we have
\eq{
    \curl(f\xi) = *\rd(f\xi) = *(f\rd\xi + \rd f\w\xi) = \frac{n+1}{2}f\xi - \nabla f\ip *\xi.
}

(5) Let $f\in P_1$. Then ${\rm Hess}f=-fg_{{\rm st}}$, and therefore
\begin{align}
    \curl (\nabla f\ip *\xi) &= -*\rd* (\rd f\w \xi) = -\rd^* (\rd f \w\xi)\\
    &= e_i\ip\nabla_{e_i}(\rd f \w\xi) = e_i\ip (\nabla_{e_i}\rd f\w\xi + \rd f\w\nabla_{e_i}\xi )\\
    &=-fe_i\ip e_i\w\xi + e_i\ip(\rd f\w e_i\ip*\xi)\\
    &=-\frac{n+1}{2}f\xi + \nabla f\ip*\xi =-\curl (f\xi),
    \end{align}
where we used Proposition~\ref{basic_form}(2) in the first equality, Proposition~\ref{basic_form}(4) in the third, and Proposition~\ref{basic_form}(6)(7) in the sixth.
\end{proof}

\section{Estimates for eigenforms}\label{sec3}

This section collects a few eigenfunction identities and establishes the sharp $L^2$-estimates needed for the stability and instability arguments in Sections~4--6.

\begin{proposition}\label{new_eigenfunction}
Let $\xi\in E_1$ and $\eta\in E_{-1}$. Then
\begin{enumerate}
    \item For any $\varphi_{k} \in E_{k}$ $(k\ge 1)$, we have $\langle \xi, \varphi_{k}\rangle \in P_{k-1}$ and $\langle \eta, \varphi_{k}\rangle \in P_{k+1}$.
    \item For any $\varphi_{-k} \in E_{-k}$ $(k\ge 1)$, we have $\langle \xi, \varphi_{-k}\rangle \in P_{k+1}$ and $\langle \eta, \varphi_{-k}\rangle \in P_{k-1}$.
\end{enumerate}
\end{proposition}

\begin{proof}
We prove the first statement in (1); the remaining cases are analogous. Using the Weitzenb\"ock formula \eqref{sphere_Weitzenbock}, we compute
\begin{align*}
    &\nabla^*\nabla\<\xi,\varphi_k\> = \<-\nabla_{e_i}\nabla_{e_i}\xi, \varphi_k\> + \<\xi, -\nabla_{e_i}\nabla_{e_i}\varphi_k\> - 2\<\nabla_{e_i}\xi, \nabla_{e_i}\varphi_k\>\\
    &= \<\Delta\xi-\frac{(n+1)(n-1)}{4}\xi, \varphi_k\> + \<\xi, \Delta\varphi_k-\frac{(n+1)(n-1)}{4}\varphi_k\> -2\<e_i\ip*\xi, \nabla_{e_i}\varphi_k\>\\
    &= \<\frac{(n+1)^2}{4}\xi-\frac{(n+1)(n-1)}{4}\xi, \varphi_k\> + \<\xi, \big(\frac{n+1}{2}+k-1\big)^2\varphi_k-\frac{(n+1)(n-1)}{4}\varphi_k\> + 2\<\xi, e_i\ip*\nabla_{e_i}\varphi_k\>\\
    &= \frac{n+1}{2}\<\xi, \varphi_k\> + \Big[ \big(\frac{n-1}{2}+k\big)^2 - \frac{(n+1)(n-1)}{4} \Big]\<\xi, \varphi_k\> - 2\<\xi, \curl\varphi_k\>\\
    &= \Big[ \frac{n+1}{2} + \big(\frac{n-1}{2}+k\big)^2 - \frac{(n+1)(n-1)}{4} - 2\big(\frac{n-1}{2}+k\big) \Big]\<\xi, \varphi_k\>\\
    &= (k-1)(n+k-2)\<\xi, \varphi_k\>,
\end{align*}
where we used Proposition~\ref{Killing_properties}(1), Proposition~\ref{basic_form}(3), and Proposition~\ref{basic_form}(2)(4).
Thus $\langle \xi,\varphi_k\rangle$ is an eigenfunction of $\nabla^*\nabla$ with eigenvalue $(k-1)(n+k-2)$, i.e. $\langle \xi,\varphi_k\rangle\in P_{k-1}$.
\end{proof}

\begin{corollary}\label{cor3.2}
Let $\xi\in E_1$. For $\varphi_k\in E_k$ and $\varphi_{-j}\in E_{-j}$ we have
\eq{
   \int \langle \xi, \varphi_k\rangle \langle \xi, \varphi_{-j}\rangle=0 \quad \hbox{for any }\ k\neq j+2.
}
\end{corollary}

We next classify the eigenspaces $E_{\pm 2}$; this is the only explicit classification needed in the main arguments. A complete classification for all $k$ is given in Appendix~\ref{secA}. See also \cite{WZ25} for an analogous picture for Dirac eigenspinors.

\begin{proposition}\label{classification}
Let $E_{\pm k}$ be the curl eigenspaces defined above. Then
\eq{
    E_2 &= {\rm span}\left\{ \frac{n+1}{2}f\xi - \nabla f\ip*\xi=\curl (f\xi) \,\Big|\, \xi\in E_1, f\in P_1 \right\},\\
    E_{-2} &= {\rm span}\left\{ \frac{n+1}{2}f\eta + \nabla f\ip*\eta=-\curl (f\eta)\,\Big|\, \eta\in E_{-1}, f\in P_1 \right\}.
}
Moreover,
\eq{
    {\rm span}\{f\xi, \nabla f\ip*\xi \,|\, \xi\in E_1, f\in P_1\} = {\rm span}\{f\xi \,|\, \xi\in E_1, f\in P_1\}.
}
\end{proposition}

\begin{proof}
We treat $E_2$; the case of $E_{-2}$ is analogous.
Define
\eq{
    V_1 &\coloneqq {\rm span}\left\{ \frac{n+1}{2}f\xi - \nabla f\ip*\xi \,\Big|\, \xi\in E_1, f\in P_1 \right\},\\
    V_2 &\coloneqq {\rm span}\left\{ f\xi + \nabla f\ip*\xi \,\Big|\, \xi\in E_1, f\in P_1 \right\},\\
    V &\coloneqq {\rm span}\{f\xi, \nabla f\ip*\xi \,|\, \xi\in E_1, f\in P_1\}.
}
By Proposition~\ref{Killing_properties}(4) we have $V_1\subset E_2$. Moreover, the same identity shows that $\curl(f\xi+\nabla f\ip*\xi)=0$, hence $V_2\subset {\rm ker}(\curl\!)={\rm ker}(\rd)$ and therefore $V_2\subset E_2^\perp$. In particular, $V=V_1\oplus V_2$.

Let $\{\xi_i\}$ be an orthonormal basis of $E_1$. For any $\varphi_2\in E_2$, Proposition~\ref{Killing_properties}(3) yields
\eq{
    \varphi_2 = \sum_i h_i\xi_i
}
for suitable functions $h_i$. By Proposition~\ref{new_eigenfunction}, we have $\langle\xi,\varphi_2\rangle\in P_1$ for every $\xi\in E_1$, and hence $h_i\in P_1$ for each $i$. It follows that $\varphi_2\in V$. Since $V_2\subset E_2^\perp$, we conclude that in fact $\varphi_2\in V_1$, and thus $E_2=V_1$. The final identity for $V$ follows immediately.
\end{proof}

We now turn to the key $L^2$-estimates that enter the second variation computations.
\begin{proposition}\label{crucial_estimate}
Let $\xi \in E_{1}$ be given with $\abs{\xi}=1$. Then
\begin{enumerate}
    \item For any $k\geq 1$ and $\varphi_{\pm k}\in E_{\pm k}$, we have
    \eq{
        \int\<\xi,\varphi_{k}\>^2 \leq \frac{n+k-2}{n+2k-3} \int \abs{\varphi_k}^2
    }
    and
    \eq{
        \int\<\xi,\varphi_{-k}\>^2 \leq \frac{k+1}{n+2k+1} \int \abs{\varphi_{-k}}^2;
    }
    \item Moreover, for $k=2$ and $\varphi_2\in E_2$ we have the sharp estimate
    \eq{
        \int\<\xi,\varphi_{2}\>^2 \leq \frac{n+1}{n+3} \int \abs{\varphi_{2}}^2,
    } 
    with equality if and only if
    \eq{
    \varphi_2 = \frac{n+1}{2}f\xi - \nabla f\ip *\xi=\curl(f\xi) \quad\hbox{for some}\quad f\in P_1.
    }
    \end{enumerate}
\end{proposition}

\begin{proof}
(1) We first consider $\varphi_k\in E_k$. The case $k=1$ is immediate. For $k\ge 2$ we have $\curl\varphi_k = \left(\frac{n-1}{2}+k\right)\varphi_k$, and hence
\begin{align}
    \int \<\xi,\varphi_k\>^2 &= \Big(\frac{n-1}{2}+k\Big)^{-1}\int \<\xi, \varphi_k\>\<\xi, \curl\varphi_k\>\\
    &= \Big(\frac{n-1}{2}+k\Big)^{-1}\int \<\curl (\<\xi, \varphi_k\>\xi), \varphi_k\>\\
    &= \Big(\frac{n-1}{2}+k\Big)^{-1}\int \Big\<\frac{n+1}{2}\<\xi,\varphi_k\>\xi - \nabla \<\xi,\varphi_k\>\ip *\xi, \varphi_k\Big\>\\
    &= \frac{\frac{n+1}{2}}{\frac{n-1}{2}+k}\int \<\xi,\varphi_k\>^2 - \frac{1}{\frac{n-1}{2}+k} \int \<\nabla \<\xi,\varphi_k\>\ip *\xi, \varphi_k\>.
\end{align}
Proposition~\ref{basic_form}(3) implies $\<\nabla \<\xi,\varphi_k\>\ip *\xi, \xi\> = 0$. Therefore,
\begin{align}
    \int \<\xi,\varphi_k\>^2 &= -\frac{1}{k-1} \int\<\nabla \<\xi,\varphi_k\>\ip *\xi, \varphi_k  \>\\
    &= -\frac{1}{k-1} \int\<\nabla \<\xi,\varphi_k\>\ip *\xi, \varphi_k - \<\xi,\varphi_k\>\xi \>\\
    &\leq \frac{1}{k-1} \Big( \int \abs{\nabla \<\xi,\varphi_k\>\ip *\xi}^2 \Big)^{\frac{1}{2}} \Big( \int \abs{\varphi_k - \<\xi,\varphi_k\>\xi }^2 \Big)^{\frac{1}{2}} \\
    &\leq \frac{1}{k-1} \Big( \int \abs{\nabla \<\xi,\varphi_k\>}^2 \Big)^{\frac{1}{2}} \Big( \int \abs{\varphi_k}^2 - \int \<\xi,\varphi_k\>^2\Big)^{\frac{1}{2}} \\
    &= \frac{1}{k-1} \Big( (k-1)(n+k-2)\int \<\xi,\varphi_k\>^2 \Big)^{\frac{1}{2}} \Big( \int \abs{\varphi_k}^2 - \int \<\xi,\varphi_k\>^2\Big)^{\frac{1}{2}}, 
\end{align}
where we used Cauchy--Schwarz' and H\"older's inequality in the first inequality, Proposition~\ref{basic_form}(1) in the second, and Proposition~\ref{new_eigenfunction}(1) in the last equality. It follows that
\eq{
    \int\<\xi,\varphi_{k}\>^2 \leq \frac{n+k-2}{n+2k-3} \int |\varphi_k|^2.
}

We next consider $\varphi_{-k}\in E_{-k}$ ($k\ge 1$). Since $\curl\varphi_{-k} = -\left(\frac{n-1}{2}+k\right)\varphi_{-k}$, we have
\begin{align}
    \int \<\xi,\varphi_{-k}\>^2 &= -\Big(\frac{n-1}{2}+k\Big)^{-1}\int \<\xi, \varphi_{-k}\>\<\xi, \curl\varphi_{-k}\>\\
    &= -\Big(\frac{n-1}{2}+k\Big)^{-1}\int \<\curl (\<\xi, \varphi_{-k}\>\xi), \varphi_{-k}\>\\
    &= -\Big(\frac{n-1}{2}+k\Big)^{-1}\int \Big\<\frac{n+1}{2}\<\xi,\varphi_{-k}\>\xi - \nabla \<\xi,\varphi_{-k}\>\ip *\xi, \varphi_{-k}\Big\>\\
    &= -\frac{\frac{n+1}{2}}{\frac{n-1}{2}+k}\int \<\xi,\varphi_{-k}\>^2 - \frac{1}{\frac{n-1}{2}+k} \int \<\nabla \<\xi,\varphi_{-k}\>\ip *\xi, \varphi_{-k}\>.
\end{align}
Proposition \ref{basic_form} (3) implies $\<\nabla \<\xi,\varphi_{-k}\>\ip *\xi, \xi\> = 0$. Hence from the previous equation we have
\begin{align}
    \int \<\xi,\varphi_{-k}\>^2 &= -\frac{1}{n+k} \int\<\nabla \<\xi,\varphi_{-k}\>\ip *\xi, \varphi_{-k} \>\\
    &= -\frac{1}{n+k} \int\<\nabla \<\xi,\varphi_{-k}\>\ip *\xi, \varphi_{-k} - \<\xi,\varphi_{-k}\>\xi \>\\
    &\leq \frac{1}{n+k} \Big( \int \abs{\nabla \<\xi,\varphi_{-k}\>\ip *\xi}^2 \Big)^{\frac{1}{2}} \Big( \int \abs{\varphi_{-k} - \<\xi,\varphi_{-k}\>\xi }^2 \Big)^{\frac{1}{2}} \\
    &\leq \frac{1}{n+k} \Big( \int \abs{\nabla \<\xi,\varphi_{-k}\>}^2 \Big)^{\frac{1}{2}} \Big( \int \abs{\varphi_{-k}}^2 - \int \<\xi,\varphi_{-k}\>^2\Big)^{\frac{1}{2}} \\
    &= \frac{1}{n+k} \Big( (k+1)(n+k)\int \<\xi,\varphi_{-k}\>^2 \Big)^{\frac{1}{2}} \Big( \int \abs{\varphi_{-k}}^2 - \int \<\xi,\varphi_{-k}\>^2\Big)^{\frac{1}{2}}, 
\end{align}
where we have used Cauchy--Schwarz' and H\"older's inequality in the first inequality, Proposition~\ref{basic_form} (1) in the second, and Proposition~\ref{new_eigenfunction} (2) in the last equality. It follows that
\eq{
    \int\<\xi,\varphi_{-k}\>^2 \leq \frac{k+1}{n+2k+1} \int |\varphi_{-k}|^2.
}

(2) Now let $\varphi_2\in E_2$. As in (1) we obtain
\eq{\label{Cauchy_E2}
    \int \<\xi,\varphi_2\>^2 &= - \int\<\nabla \<\xi,\varphi_2\>\ip *\xi, \varphi_2 - \<\xi,\varphi_2\>\xi \>\\
    &\leq \Big( \int \abs{\nabla \<\xi,\varphi_2\>\ip *\xi}^2 \Big)^{\frac{1}{2}} \Big( \int \abs{\varphi_2 - \<\xi,\varphi_2\>\xi }^2 \Big)^{\frac{1}{2}}.
}
Since $\<\xi,\varphi_2\>\in P_1$, we compute
\begin{align}
    \int \abs{\nabla \<\xi,\varphi_2\>\ip *\xi}^2 &= \int \abs{\rd \<\xi, \varphi_2\> \w\xi}^2= \int \<\rd \<\xi, \varphi_2\> \w\xi, \rd \<\xi, \varphi_2\> \w\xi\>\\
    &=\int \<\rd \<\xi, \varphi_2\>\w\xi, \rd (\<\xi, \varphi_2\>\xi) - \frac{n+1}{2}\<\xi, \varphi_2\>*\xi\>\\
    &= \int \< \rd^*(\rd \<\xi, \varphi_2\> \w\xi), \<\xi, \varphi_2\>\xi\> = \int \< * \rd *(\rd \<\xi, \varphi_2\> \w\xi), \<\xi, \varphi_2\>\xi\>\\
    &= \int \< \curl (-\nabla \<\xi, \varphi_2\>\ip *\xi), \<\xi, \varphi_2\>\xi\> = \int \< \frac{n+1}{2}\<\xi, \varphi_2\>\xi - \nabla \<\xi, \varphi_2\>\ip *\xi, \<\xi, \varphi_2\>\xi\>\\
    &= \frac{n+1}{2}\int \<\xi, \varphi_2\>^2,
\end{align}
where we used Proposition~\ref{basic_form}(2) in the first and sixth equalities, Proposition~\ref{Killing_properties}(4) in the seventh, and Proposition~\ref{basic_form}(3) in the last.
Combining this with \eqref{Cauchy_E2} yields
\eq{
    \int \<\xi,\varphi_2\>^2 &\leq \Big( \frac{n+1}{2}\int \<\xi, \varphi_2\>^2 \Big)^{\frac{1}{2}} \Big( \int \abs{\varphi_2}^2 - \int \<\xi, \varphi_2\>^2 \Big)^{\frac{1}{2}}.
}
Hence
\eq{
        \int\<\xi,\varphi_{2}\>^2 \leq \frac{n+1}{n+3} \int \abs{\varphi_{2}}^2.
}
The only inequality is \eqref{Cauchy_E2}, where we applied Cauchy--Schwarz' and H\"older's inequality. Thus equality holds if and only if
\eq{
    \varphi_2 - \<\xi,\varphi_2\>\xi = -C\, \nabla \<\xi,\varphi_2\>\ip *\xi \quad\hbox{for some constant}\ C>0.
}
By Proposition~\ref{classification} we have $C=\frac{2}{n+1}$. Taking $f = \frac{2}{n+1}\<\xi,\varphi_2\>\in P_1$ completes the proof.
\end{proof}

\section{Local stability}\label{sec4}

We begin by recalling the Euler--Lagrange equation for $J_1$:
\eq{
    \curl \big( \abs{ \curl\alpha }^{-\frac{2}{n+1}} \, \curl\alpha \big) = \mu\, \curl\alpha \quad\hbox{for some constant}\ \mu.
}
Equivalently, if we set $\beta \coloneqq \abs{ \curl\alpha }^{-\frac{2}{n+1}} \curl\alpha$, then $\beta$ satisfies the conformally invariant equation
\eq{
    \curl \beta =\mu \, \abs{\beta}^{\frac{2}{n-1}} \beta.
}
Positive and negative Killing $\frac{n-1}{2}$-forms solve this equation, and by conformal invariance so do their conformal images. The goal of this section is to show that this conformal family is locally stable for $J_1$.

\medskip
We parametrize conformal maps of $\S^n$ as follows. For $b\in \mathbb{B}^{n+1}$ with $\abs{b}<1$ define
\eq{\label{eq_conformal_transformation}
    \Xi_b(x) = \frac{(1-|b|^2)x+2(1+\langle b,x\rangle)b}{1+2\langle b,x\rangle+|b|^2}.
}
One easily checks
that $\Xi_b$ is conformal and
\eq{
    \langle (\Xi_{b})_* (v), (\Xi_{b})_* (w) \rangle = \left( \frac {1-|b|^2 }{ (1+\langle x, b\rangle)^2} \right)^2 \langle v, w\rangle,
    \qquad
    (\det D\Xi_{b})^ {\frac 1n} = \frac  {1-|b|^2 }{ (1+\langle x, b\rangle)^2}.
}
In particular, the conformal family generated by positive Killing forms can be written as
\eq{\label{explicit_M}
    {\mathcal M}= \left\{(\Xi_{b})^*\xi \, |\, 0\not\equiv\xi \in E_1,\ b\in \mathbb{B}^{n+1}\right\}.
}

By conformal invariance it suffices to analyze the second variation of $J_1$ at a fixed normalized positive Killing form. Fix $\xi\in E_1$ with $\abs{\xi}=1$.

For later use, define the linear map
\eq{\label{Phi}
    \Phi_\xi \coloneqq  \Phi_\xi (f) \coloneqq \frac{n+1}{2} f \xi - \nabla f\ip *\xi \quad \hbox{for}\quad f\in P_1.
}
By Proposition~\ref{classification}, $\Phi_\xi(f)\in E_2$ for every $f\in P_1$. We set
\eq{
    Q\coloneqq Q_\xi\coloneqq   \{\Phi_\xi(f)\,|\, f\in P_1\} \subset E_2.
}

Finally, note that $J_1$ depends only on $\curl\alpha=*\rd\alpha$, and hence
\eq{
    J_1(\alpha) = J_1(\alpha + \beta) \quad\hbox{for any closed}\ \beta.
}
Thus $J_1$ (and its variations) are naturally defined on the quotient space $\Omega^{\frac{n-1}{2}}/\,{\rm ker}(\rd)$, and we will freely replace a class $[\alpha]$ by a convenient representative.

To identify the nondegenerate directions in the second variation, we first describe the tangent space of the conformal family $\mathcal M$.

\begin{lemma}\label{tangent_space}
Let $\xi\in E_1$. Modulo closed forms, the tangent space of $\mathcal M$ at $\xi$ is
\eq{
    T_\xi\mathcal{M} = E_1 \oplus Q_\xi \quad ({\rm mod}\ {\rm ker}(\rd)).
}
\end{lemma}
\begin{proof}
Choose a variation of the conformal parameter of the form $b(t)=t e_i$. Differentiating at $t=0$, one checks
\eq{
    \frac{\rd}{\rd t}\Big|_{t=0}\left(\frac  {1-|b|^2 }{ (1+\langle x, b\rangle)^2} \right)^{\frac {1}{2}} = -2x_i,
}
and
\eq{
    \frac{\rd}{\rd t}\Big|_{t=0}\Xi_b(x) = e_i - x_ix = 2\nabla x_i,
}
where $\nabla x_i$ is a conformal Killing vector field on $\S^n$. Therefore, at $\xi\in E_1$ the tangent space $T_\xi \mathcal{M}$ is generated (modulo ${\rm ker}(\rd)$) by the infinitesimal conformal actions
\eq{
    \frac{\rd}{\rd t}\Big|_{t=0}(\Xi_{b})^*\xi = 2\mathcal{L}_{\nabla x_i}\xi,
}
together with the directions in $E_1$.

Let $f \coloneqq 2x_i\in P_1$. By Proposition~\ref{basic_form}(4)(5) we compute
\begin{align}
    \mathcal{L}_{\nabla f}\xi &= \rd (\nabla f \ip\xi) + \nabla f \ip \rd\xi\\
    &= e_i^\flat\w\nabla_{e_i}(\nabla f \ip\xi) + \frac{n+1}{2}\nabla f\ip*\xi\\
    &= e_i^\flat\w(\nabla_{e_i}\nabla f \ip\xi + \nabla f \ip \nabla_{e_i}\xi) + \frac{n+1}{2}\nabla f\ip*\xi\\
    &= e_i^\flat\w(-f e_i\ip\xi + \nabla f \ip e_i\ip*\xi) + \frac{n+1}{2}\nabla f\ip*\xi\\
    &= -\frac{n-1}{2}f\xi + \nabla f \ip*\xi - \nabla f\ip e_i^\flat\w e_i\ip*\xi + \frac{n+1}{2}\nabla f\ip*\xi\\
    &= -\frac {n-1} 2 f \xi + \nabla f \ip* \xi, \label{vectors}
\end{align}
where we used Proposition~\ref{basic_form}(7) in the fifth equality and Proposition~\ref{basic_form}(6) in the fifth and last.

Finally, each direction in \eqref{vectors} admits the decomposition
\eq{
    -\frac {n-1} 2 f \xi + \nabla f \ip* \xi = \frac{2}{n+3} \left[ -\frac{n+1}{2} \Big( \frac{n+1}{2}f\xi - \nabla f \ip *\xi \Big) + (f\xi + \nabla f \ip *\xi )\right] \in Q_\xi \oplus V_2,
}
where $V_2\subset {\rm ker}(\rd)$ as shown in the proof of Proposition~\ref{classification}. This proves that, modulo closed forms, the tangent space is $E_1\oplus Q_\xi$.
\end{proof}

\begin{corollary}
\label{tangent_perp}
As a consequence, we have
\eq{
     T_\xi\mathcal{M}^\perp &= {\rm ker}(\rd)^\perp \cap (E_1 \oplus Q_\xi)^\perp \\
    &= (E_2 \cap Q_\xi^\perp) \oplus (E_{-1}\oplus E_{3}) \oplus (E_{-2}\oplus E_{4}) \oplus \cdots \oplus
    (E_{-k}\oplus E_{k+2}) \oplus\cdots
}
\end{corollary}

\begin{proposition}\label{second_variation}
The (formal) second variation of $J_1$ on standard $\S^n$ at $\xi\in E_{1}$ (with normalization $|\xi|=1$) is given by
\eq{
  \frac{\rd^2}{\rd t^2}\Big|_{t=0}J_1(\xi+t\varphi) = 2\omega_{n}^{\frac{1-n}{n}}\Bigg\{ \omega_n^{-1}\Big( \int \<\xi, \varphi\> \Big)^2 + \frac{2}{n+1} \int \abs{\curl \varphi}^2 - \frac{4}{(n+1)^2} \int \<\xi, \curl \varphi \>^2 - \int \<\curl\varphi, \varphi\> \Bigg\}.
   }
\end{proposition}
\begin{proof}
The computation is elementary; we include it for completeness.

In general, for a quotient functional $J=U/V$, the Euler--Lagrange equation is $U'V-UV'=0$. At a critical point this gives the second variation formula
\eq{
    J'' = \frac{U''V-UV''}{V^2}.
}
Here
\eq{
    U(\alpha)=\Big(\int\abs{\curl\alpha}^{\frac{2n}{n+1}}\Big)^{\frac{n+1}{n}},\quad V(\alpha)=\int\<\curl\alpha,\alpha\>.
}
A direct computation of the second variations of $U$ and $V$ at $\xi\in E_{1}$ yields
\eq{
    U''(\xi)(\varphi,\varphi) &= (n+1)\omega_{n}^{\frac{1-n}{n}}\Big(\int\<\xi,\varphi\>\Big)^2 - \frac{4}{n+1}\omega_{n}^{\frac{1}{n}}\int\<\xi,\curl\varphi\>^2 + 2\omega_{n}^{\frac{1}{n}}\int\abs{\curl\varphi}^2,\\
    V''(\xi)(\varphi,\varphi) &= 2\int\<\curl\varphi,\varphi\>.
}
Together with
\eq{
    U(\xi) = \frac{(n+1)^2}{4}\omega_{n}^{\frac{1+n}{n}},\quad V(\xi) = \frac{n+1}{2}\omega_{n},
}
this gives the stated formula.
\end{proof}

We now prove Theorem~\ref{thm1.1}. The key input is a spectral gap for the second variation of $J_1$ in directions transverse to the conformal symmetry manifold $\mathcal M$.

\begin{proof}[Proof of Theorem~\ref{thm1.1}]
By conformal invariance, it suffices to consider the case where $\alpha$ is close (in the $\curl$-norm appearing in the statement) to a fixed normalized positive Killing form $\xi\in E_1$.
After modulating by the conformal parameters, i.e. projecting away the tangent directions $T_\xi\mathcal M$, the quadratic form given by the second variation of $J_1$ is strictly positive by Theorem~\ref{spectral_gap} below.
A standard quantitative argument (following \cite{Figalli_Zhang_20}) then upgrades this spectral gap to the nonlinear estimate in Theorem~\ref{thm1.1}; in our setting the implementation is identical to \cite[Appendix~A]{WZ25}.
\end{proof}

\begin{theorem}\label{spectral_gap}
For any $\xi\in E_1$ with $\abs{\xi}=1$, there exists $c(n)>0$, such that for any $\varphi \in W^{1,2}$ with  $\varphi\in T_\xi\mathcal{M}^{\perp}$ we have
\eq{
    \frac{2}{n+1} \int \abs{\curl \varphi}^2 - \frac{4}{(n+1)^2} \int \<\xi, \curl \varphi \>^2 - \int \<\curl\varphi, \varphi\> \geq c(n)\int\abs{\curl\varphi}^2.
}
It implies that there exists $c_1(n)>0$ such that
\eq{
    \delta^2 J_1 (\xi) (\varphi, \varphi) \ge c_1(n)\int | \curl \varphi|^2, \quad\forall \varphi\in T_\xi\mathcal{M}^{\perp}.
}
\end{theorem}
\begin{proof}
Define the quadratic form
\eq{
    G(\varphi) \coloneqq \frac{2}{n+1} \int \abs{\curl \varphi}^2 - \frac{4}{(n+1)^2} \int \<\xi, \curl \varphi \>^2 - \int \<\curl\varphi, \varphi\>.
}
Since $\curl$ is self-adjoint, $G$ depends only on the $L^2$-projection of $\varphi$ onto ${\rm ker}(\rd)^\perp$, i.e.
$G(\varphi)=G({\rm proj}_{{\rm ker}(\rd)^\perp}(\varphi))$.

By Lemma~\ref{tangent_space} and Corollary~\ref{tangent_perp} we have the orthogonal decomposition
\eq{
    L^2(\Omega^{\frac{n-1}{2}}) = {\rm ker}(\rd) \oplus (E_1 \oplus Q) \oplus (E_2 \cap Q^\perp) \oplus (E_{-1}\oplus E_{3}) \oplus (E_{-2}\oplus E_{4}) \oplus \cdots,
}
where $Q=Q_\xi$. Moreover, by Corollary~\ref{cor3.2} it suffices to estimate $G$ on each summand separately, i.e.
\eq{
    G = G|_{E_2 \cap Q^\perp} + G|_{E_{-1}\oplus E_{3}} + G|_{E_{-2}\oplus E_{4}} + \cdots.
}

For $k\ge 2$ and $\varphi\in E_{-k}\oplus E_{k+2}$, Cauchy--Schwarz gives
\eq{
    G(\varphi) &\geq \frac{2}{n+1} \int \abs{\curl \varphi}^2 - \frac{4}{(n+1)^2} \int \abs{\curl \varphi}^2 - \Big( \frac{n+1}{2} + k + 1 \Big)^{-1} \int \<\curl\varphi, \varphi\>\\
    &\geq \Big( \frac{2}{n+1} - \frac{4}{(n+1)^2} - \frac{2}{n+7} \Big)\int \abs{\curl \varphi}^2\\
    &= \frac{8(n-2)}{(n+1)^2(n+7)} \int \abs{\curl \varphi}^2 \eqcolon c_{1}(n)\int \abs{\curl \varphi}^2.
}

Next, consider $\varphi\in E_2 \cap Q^\perp$. Using Proposition~\ref{crucial_estimate}(2), we obtain
\eq{
    G(\varphi) &= \frac{2}{n+1}\cdot \frac{(n+3)^2}{4} \int \abs{\varphi}^2 - \frac{4}{(n+1)^2}\cdot \frac{(n+3)^2}{4} \int \<\xi, \varphi \>^2 - \frac{n+3}{2} \int \<\xi, \varphi \>^2\\
    &\geq \frac{(n+3)^2}{2(n+1)} \int \abs{\varphi}^2 - \frac{(n+3)^2}{(n+1)^2}\cdot \frac{n+1}{n+3} \int \abs{\varphi}^2 - \frac{n+3}{2} \int \abs{\varphi}^2 = 0,
}
with equality if and only if
\eq{
    \varphi_2 = \frac{n+1}{2}f\xi - \nabla f\ip *\xi
}
for some $f\in P_1$. In particular, $G(\varphi)\ge 0$, and $G(\varphi)=0$ holds if and only if $\varphi\in Q$. Therefore $G(\varphi)>0$ on $E_2\cap Q^\perp$. Since $E_2$ is finite-dimensional and $G$ is quadratic, there exists $c_{2}(n)>0$ such that
\eq{
    G(\varphi)\geq c_{2}(n)\int \abs{\varphi}^2, \quad \forall \,\varphi \in E_2\cap Q^\perp.
}

Finally, consider $\varphi\in E_{-1}\oplus E_{3}$ and decompose
\eq{
    \varphi = -\frac{2}{n+1}\varphi_{-1} + \frac{2}{n+5}\varphi_{3}
}
so that
\eq{
    \curl \varphi = \varphi_{-1} + \varphi_{3}.
}
Using Proposition~\ref{crucial_estimate}(1) we compute
\begin{align}
    G(\varphi) &= \frac{2}{n+1} \int \abs{\curl \varphi}^2 - \frac{4}{(n+1)^2} \int \<\xi, \curl \varphi \>^2 - \int \<\curl\varphi, \varphi\>\\
    &= \frac{2}{n+1} \left( \int \abs{\varphi_{-1}}^2 + \int \abs{\varphi_{3}}^2 \right) - \frac{4}{(n+1)^2} \int (\<\xi,\varphi_{-1}\> + \<\xi,\varphi_{3}\>)^2 + \frac{2}{n+1}\int \abs{\varphi_{-1}}^2 - \frac{2}{n+5} \int \abs{\varphi_{3}}^2\\
    &= \frac{4}{n+1} \int \abs{\varphi_{-1}}^2 + \frac{8}{(n+1)(n+5)} \int \abs{\varphi_{3}}^2 - \frac{4}{(n+1)^2} \int (\<\xi,\varphi_{-1}\> + \<\xi,\varphi_{3}\>)^2 \\
    &\geq \frac{4}{n+1}\cdot \frac{n+3}{2} \int \<\xi,\varphi_{-1}\>^2 + \frac{8}{(n+1)(n+5)}\cdot \frac{n+3}{n+1} \int \<\xi,\varphi_{3}\>^2 - \frac{4}{(n+1)^2} \int (\<\xi,\varphi_{-1}\> + \<\xi,\varphi_{3}\>)^2 \\
    &= \frac{2(n^2+4n+1)}{(n+1)^2} \int \<\xi,\varphi_{-1}\>^2 + \frac{4}{(n+1)(n+5)} \int \<\xi,\varphi_{3}\>^2 - \frac{8}{(n+1)^2} \int \<\xi,\varphi_{-1}\>\<\xi,\varphi_{3}\>.
\end{align}
It is then straightforward to conclude that $G(\varphi)\geq c_{3}(n)\int\abs{\curl\varphi}^2$ for some $c_{3}(n)>0$.

Taking $c(n)\coloneqq \min\{c_{1}(n),c_{2}(n),c_{3}(n)\}>0$ completes the proof.
\end{proof}

\section{The second form-Sobolev inequality}\label{sec5}

We now turn to the second Sobolev quotient
\eq{
    J_2(\alpha) \coloneqq \frac{\Big(\int \abs{\curl \alpha}^{\frac{2n}{n+1}} \,\rdV\Big)^{\frac{n+1}{n}}}{\inf_{\phi\in\Omega^{\frac{n-3}{2}}} \Big(\int \abs{\alpha - \rd\phi}^{\frac{2n}{n-1}} \,\rdV\Big)^{\frac{n-1}{n}}},
    \qquad \alpha\not\equiv 0,
}
with sharp constant $S_2\coloneqq \inf_{\alpha\not\equiv 0} J_2(\alpha)$.
Equivalently, after replacing $\alpha-\rd\phi$ by $\alpha$, one may write
\eq{
    \inf_{\rd^*(\abs{\alpha}^{\frac{2}{n-1}}\alpha)=0} \frac{\Big(\int \abs{\curl \alpha}^{\frac{2n}{n+1}} \,\rdV\Big)^{\frac{n+1}{n}}}{ \Big(\int \abs{\alpha}^{\frac{2n}{n-1}} \,\rdV\Big)^{\frac{n-1}{n}}} = S_2,
}
where the constraint is exactly the Euler--Lagrange equation for the minimization problem in the denominator. In particular, any critical point of $J_2$ satisfies
\eq{
    \rd^*\left(\abs{ \alpha - \rd\phi}^{\frac{2} {n-1}} (\alpha - \rd\phi)\right) = 0.
}

Frank--Loss~\cite{FL22} proved existence of minimizers for $n=3$ (and their method extends to all $n\ge 3$). Moreover, each positive Killing form $\xi\in E_1$ solves the Euler--Lagrange equation \eqref{eq_Y2'}; it is therefore natural to ask whether
\eq{
    J_2(\alpha) \geq  J_2(\xi), \quad \forall \alpha \not\equiv 0.
}
We show that this conjecture fails: for every $n\equiv 3\!\mod 4$, the infimum of $J_2$ is strictly smaller than $J_2(\xi)$.

\medskip
We first record the value of $J_2(\xi)$. Since $\abs{\xi}$ is constant, $\rd\phi=0$ is a critical point of $\rd\phi\mapsto \norm{\xi-\rd\phi}_{\frac{2n}{n-1}}$, and by convexity of the $L^{\frac{2n}{n-1}}$-norm we have
\eq{
    \inf_{\phi} \Big(\int \abs{\xi - \rd\phi}^{\frac{2n}{n-1}}\Big)^{\frac{n-1}{n}} = \Big(\int \abs{\xi}^{\frac{2n}{n-1}}\Big)^{\frac{n-1}{n}}.
}
Consequently,
\eq{
    J_2(\xi) = \frac{(n+1)^2}{4}\omega_n^{2/n}.
}

\begin{proposition}\label{prop5.1}
Given any $0\not\equiv\xi\in E_1$, there exists $\varphi_{-1}\in E_{-1}$ such that
\eq{
    J_2(\xi + \varphi_{-1}) < \frac{(n+1)^2}{4}\omega_n^{2/n}.
}
\end{proposition}

\begin{proof}
Let $\alpha=\xi+\varphi_{-1}$ with $0\not\equiv\varphi_{-1}\in E_{-1}$. Since $E_1\perp E_{-1}$ in $L^2$, we have
\eq{
    \int \< \xi, \varphi_{-1} \> =0.
}
Moreover,
$\curl \alpha =\curl \xi +\curl \varphi_{-1}=\frac{n+1}{2} (\xi-\varphi_{-1})$.

For the numerator, H\"older's inequality gives
\eq{
    \Big( \int \abs{ \curl\alpha }^{\frac{2n}{n+1}} \Big)^{\frac{n+1}{n}}
    \leq \omega_{n}^{\frac{1}{n}} \int \abs{\curl\alpha}^2
    = \frac{(n+1)^2}{4}\omega_{n}^{\frac{1}{n}} \int \abs{\xi-\varphi_{-1}}^2
    = \frac{(n+1)^2}{4}\omega_{n}^{\frac{1}{n}}\Big(\int \abs{\xi}^2 + \int \abs{\varphi_{-1}}^2 \Big).
}

For the denominator, applying H\"older again yields
\eq{
    \inf_{\phi} \Big(\int \abs{\alpha - \rd\phi}^{\frac{2n}{n-1}}\Big)^{\frac{n-1}{n}} \geq \omega_{n}^{-\frac{1}{n}} \inf_{\phi} \int \abs{\alpha - \rd\phi}^2.
}
Since $\alpha$ is co-closed, the Hodge decomposition \eqref{sphere_Hodge} gives $\alpha=\rd^*\beta$ for some $\beta\in\Omega^{\frac{n+1}{2}}$. Therefore,
\eq{
    \inf_{\phi} \int \abs{\alpha - \rd\phi}^2
    = \inf_{\phi} \int \big( \abs{\rd^*\beta}^2 + \abs{\rd \phi}^2 \big)
    = \int \abs{\alpha}^2.
}
Consequently,
\eq{
    \inf_{\phi} \Big(\int \abs{\alpha - \rd\phi}^{\frac{2n}{n-1}}\Big)^{\frac{n-1}{n}}
    \geq \omega_{n}^{-\frac{1}{n}} \int \abs{\alpha}^2
    = \omega_{n}^{-\frac{1}{n}}\Big(\int \abs{\xi}^2 + \int \abs{\varphi_{-1}}^2 \Big).
}
Combining these estimates yields
\eq{\label{min}
    J_2(\alpha) \le \frac{(n+1)^2}{4} \omega_n^{\frac 2 n}.
}

We now characterize the equality case. Equality holds if and only if both $\abs{\xi+\varphi_{-1}}^2$ and $\abs{\xi-\varphi_{-1}}^2$ are constant, equivalently $\<\xi,\varphi_{-1}\>$ is constant. By Proposition~\ref{new_eigenfunction}, $\<\xi,\varphi_{-1}\>\in P_2$, so equality in \eqref{min} holds if and only if $\<\xi,\varphi_{-1}\>\equiv 0$.

It remains to show that we can choose $\varphi_{-1}\in E_{-1}$ such that $\<\xi,\varphi_{-1}\>\not\equiv 0$. Using Proposition~\ref{Killing_properties}(3) (as in the proof of Proposition~\ref{classification}), we may expand
\eq{
    \xi = \sum_{l,i}c_{l,i}\,f_l\,\eta_i,
}
where $\{\eta_i\}$ is an orthonormal basis of $E_{-1}$ and $\{f_l\}$ is an orthonormal basis of $P_2$.
If $\<\xi,\varphi_{-1}\>\equiv 0$ for all $\varphi_{-1}\in E_{-1}$, then in particular for each basis element $\eta_j$ we would have
\eq{
    0 = \<\xi,\eta_j\> = \sum_l c_{l,j} f_l,
}
forcing $c_{l,j}=0$ for all $l,j$, a contradiction.
Thus there exists $\varphi_{-1}\in E_{-1}$ with $\<\xi,\varphi_{-1}\>\not\equiv 0$, and for this choice the inequality in \eqref{min} is strict.
\end{proof}

The argument above is inspired by a similar idea in \cite{WZ25}. By conformal invariance, Proposition~\ref{prop5.1} implies Theorem~\ref{thm1.2}.

\begin{remark}
(1) In view of \cite{WZ25}, we expect the Morse index of $\xi$ to be $n+1$.

(2) When $n=3$, the instability for $J_2$ is closely related to the instability for the conformally invariant 1-form quotient studied in Section~6; in that case one can identify an explicit choice of $\varphi_{-1}$.
\end{remark}

\section{A family of generalized inequalities}\label{sec6}

In this section we discuss a broader family of Sobolev-type quotients introduced by Frank--Loss \cites{FL22, FL1}. Let $n\ge 3$ be odd and let $\alpha$ be a $k$-form. For $1<p<n$ define
\eq{\label{L_pq}
    J_{p,k}(\alpha) \coloneqq \frac {\big( \int \abs{\curl \alpha}^{p} \,\rdV\big)^{q/p}} {\int \abs{\alpha}^q \,\rdV},
    \qquad q=\frac{np}{n-p}.
}
By \eqref{conformal_change}, the quotient $J_{p,k}$ is conformally invariant if and only if $p=\frac{n}{k+1}$.
Two conformally invariant choices will be particularly relevant here:
\begin{itemize}
\item $k=1$ and $p=\frac n2$ (with $n$ odd);
\item $k=\frac{n-1}{2}$ and $p=\frac{2n}{n+1}$ (with $n\equiv 3\!\mod 4$), which recovers $J_2$ from Section~5.
\end{itemize}
We first treat the case $k=1$, $p=\frac n2$, and prove Theorem~\ref{thm1.3}.

\subsection{Proof of Theorem \ref{thm1.3}}
For $k=1$ and $p=\frac n2$ (with $n$ odd), the operator $*\rd$ is (formally) self-adjoint on $1$-forms, and one is led to the conformally invariant inequality on $\R^n$
\eq{\label{A1}
    \inf_{\substack{\alpha \in \Omega^1, \\ \rd^*(\abs{\alpha}^{n-2}\alpha)=0}}
    J_{\frac{n}{2},1}(\alpha)
    =
    \inf_{\substack{\alpha \in \Omega^1, \\ \rd^*(\abs{\alpha}^{n-2}\alpha)=0}}
    \frac {\big( \int \abs{\curl \alpha}^{\frac{n}2}  \,\rdV\big)^2}{\int \abs{\alpha}^n \,\rdV}
    = S_3.
}
As in the $J_2$-problem, the constraint arises from minimizing over exact perturbations in the denominator.
The corresponding Euler--Lagrange equation is
\eq{\label{EL_another}
    \curl\big(\abs{\curl\alpha}^{\frac{n}{2}-2}\curl\alpha\big) = \mu_3 \, \abs{\alpha}^{n-2}\alpha \quad\hbox{for some constant}\ \mu_3>0.
}
As for $J_2$, any solution of \eqref{EL_another} automatically satisfies $\rd^*(\abs{\alpha}^{n-2}\alpha)=0$.

A distinguished solution of \eqref{EL_another} is provided by the Loss--Yau construction  \cite{Loss_Yau_86} (and its higher-dimensional generalization due to Dunne--Min \cite{Dunne_Min_08}). Concretely, on $\R^n$ define
\eq{\label{construction_xi}
    \xi = A^\flat = \left[n\left( \frac 1 {1+|x|^2} \right)^2\big((1-|x|^2) e_1 +2 x_1 x +2\Sigma x\big)\right]^\flat,
}
where
\eq{
    \Sigma \coloneqq \left ( \begin{matrix}
    0 &  &  &  &  &  \\
    & 0 & -1 &  &  &  \\
    & 1 & 0  &  &  &  \\
    &  &  & \ddots &  &  \\
    &  &  &  & 0 & -1 \\
    &  &  &  & 1 & 0
    \end{matrix}\right).
}
For $n=3$, the vector field $A$ was discovered in \cite{Loss_Yau_86} as the first example of a zero mode; Dunne--Min \cite{Dunne_Min_08} extended this construction to higher odd dimensions. Such solutions play an important role in the study of zero modes; see \cites{Loss_Yau_86, Dunne_Min_08, Frank_Loss_2024}.

Frank--Loss conjectured that the infimum $S_3$ in \eqref{A1} is achieved by $\xi$. For $n=3$ this is false by the example in Section~5. We show that the conjecture is also false for every odd $n$.

To exploit conformal invariance, we denote the form in \eqref{construction_xi} by $\xi_{\R^n}$ and let $\xi_{\S^n}$ be its stereographic pullback to $\S^n$ with respect to
\eq{
    g_{\S^n} = \left(\frac{2}{1+\abs{x}^2}\right)^2 g_{\mathbb{R}^n}.
}
We also introduce the reflected ``twin'' 1-form
\eq{
    \bar{\xi}_{\R^n} = \left[n\left( \frac 1 {1+|x|^2} \right)^2\big((1-|x|^2) e_1 +2 x_1 x -2\Sigma x\big)\right]^\flat,
}
and denote its stereographic pullback by $\bar\xi_{\S^n}$. When there is no ambiguity we omit the subscript $g_{\S^n}$ in norms and inner products.

\begin{lemma}\label{lem6.1}
On $\S^n$ we have
\begin{enumerate}
    \item $\abs{\xi} = \abs{\bar{\xi}} \equiv \frac{n}{2}$;
    \item $\abs{\curl\xi} = \abs{\curl\bar{\xi}} \equiv n\big(\frac{n-1}{2}\big)^{1/2}$;
    \item $\xi$ and $\bar{\xi}$ are solutions to \eqref{EL_another}.
\end{enumerate}
\end{lemma}

\begin{proof}
We sketch the proof; the details follow from direct (though somewhat lengthy) computations.
Since the expressions simplify on $\S^n$ while derivatives are easier to compute on $\R^n$, we start from $\xi_{\R^n}$ and then use conformal covariance to transfer the identities to $\S^n$.
First, we have
\eq{\label{lem6.1_eq0}
    \abs{\xi_{\mathbb{R}^n}}_{g_{\mathbb{R}^n}}^2 = \frac{n^2}{(1+\abs{x}^2)^2}.
}   
The derivative of $\xi_{\mathbb{R}^n}$ is
\eq{
    \nabla_{e_i}\xi_{\mathbb{R}^n} = -\frac{4nx_i}{(1+\abs{x}^2)^3}\Big((1-\abs{x}^2)e_1 + 2x_1x + 2\Sigma x\Big)^\flat + \frac{n}{(1+\abs{x}^2)^2}\Big( -2x_ie_1 + 2x_1e_i + 2\delta_{1i}x + 2\Sigma e_i\Big)^\flat.
}
Using Proposition \ref{basic_form} (4) we have
\eq{
    \rd\xi_{\mathbb{R}^n} = -\frac{4n}{(1+\abs{x}^2)^3}\Big( 2x^\flat\w e_1^\flat + 2x^\flat\w (\Sigma x)^\flat - (1+\abs{x}^2)(e_2^\flat\w e_3^\flat + \cdots e_{n-1}^\flat \w e_n^\flat) \Big).
}
Hence
\eq{\label{lem6.1_eq01}
    \abs{*_{g_{\mathbb{R}^n}}\rd\xi_{\mathbb{R}^n}}_{g_{\mathbb{R}^n}}^2 = \abs{\rd\xi_{\mathbb{R}^n}}_{g_{\mathbb{R}^n}}^2 = \frac{8n^2(n-1)}{(1+\abs{x}^2)^4}.
}
It follows (for short we omit the subscript $g_{\mathbb{R}^n}$ in this step)
\eq{\label{lem6.1_eq1}
    &\curl(\abs{\curl\xi}^{\frac{n}{2}-2}\curl\xi) = *\rd(\abs{*\rd\xi}^{\frac{n}{2}-2}*\!\rd\xi)\\
    =& *\!\rd\left\{\frac{c}{(1+\abs{x}^2)^{n-1}}\Big( 2*(x^\flat\w e_1^\flat) + 2*(x^\flat\w (\Sigma x)^\flat) - (1+\abs{x}^2)*(e_2^\flat\w e_3^\flat + \cdots e_{n-1}^\flat \w e_n^\flat) \Big)\right\}\\
    =&-\frac{2(n-1)c}{(1+\abs{x}^2)^n} *\left\{ x^\flat\w \Big( 2*(x^\flat\w e_1^\flat) + 2*(x^\flat\w (\Sigma x)^\flat) - (1+\abs{x}^2)*(e_2^\flat\w e_3^\flat + \cdots e_{n-1}^\flat \w e_n^\flat) \Big) \right\}\\
    &+\frac{c}{(1+\abs{x}^2)^{n-1}}\left\{ 2\rd^*(x^\flat\w e_1^\flat) + 2\rd^*(x^\flat\w(\Sigma x)^\flat) - \rd^*\big((1+\abs{x}^2)(e_2^\flat\w e_3^\flat + \cdots e_{n-1}^\flat \w e_n^\flat)\big) \right\}.
}
Every term can be computed using Proposition \ref{basic_form}. For example
\eq{
    *\big( x^\flat\w *(x^\flat\w e_1^\flat) \big) = -x\ip x^\flat\w e_1^\flat = -\abs{x}^2e_1^\flat + x^\flat\w x\ip e_1^\flat = -\abs{x}^2e_1^\flat + x_1x^\flat,
}
and
\eq{
    \rd^*(x^\flat\w e_1^\flat) = -e_i\ip\nabla_{e_i}(x^\flat\w e_1^\flat) = -e_i\ip e_i^\flat\w e_1^\flat = -(n-1)e_1^\flat.
}
Similarly we have
\eq{
    *\big( x^\flat\w *(x^\flat\w(\Sigma x)^\flat) \big) = -\abs{x}^2\Sigma x, \quad *\big( x^\flat\w *(e_2^\flat\w e_3^\flat) \big) = -x_2e_3^\flat + x_3e_2^\flat,
}
and
\eq{
    \rd^*(x^\flat\w (\Sigma x)^\flat) = -n(\Sigma x)^\flat, \quad \rd^*\big( (1+\abs{x}^2) e_2^\flat\w e_3^\flat \big) = -2x_2e_3^\flat + 2x_3e_2^\flat.
}
Hence
\eq{
    \rd^*\big((1+\abs{x}^2)(e_2^\flat\w e_3^\flat + \cdots e_{n-1}^\flat \w e_n^\flat)\big) = -2 (\Sigma x)^\flat.
}
Together with \eqref{lem6.1_eq1} it follows
\eq{
    \curl(\abs{\curl\xi}^{\frac{n}{2}-2}\curl\xi) = -\frac{2(n-1)c}{(1+\abs{x}^2)^n} \Big( (1-\abs{x}^2)e_1 + 2x_1x + 2\Sigma x \Big)^\flat.
}
Now we revert the subscript $g_{\mathbb{R}^n}$. Recall \eqref{lem6.1_eq0}, and we have
\eq{
    \curl_{g_{\mathbb{R}^n}}(\abs{\curl_{g_{\mathbb{R}^n}}\xi_{\mathbb{R}^n}}_{g_{\mathbb{R}^n}}^{\frac{n}{2}-2}\curl_{g_{\mathbb{R}^n}}\xi_{\mathbb{R}^n}) = \mu_3 \abs{\xi_{\mathbb{R}^n}}_{g_{\mathbb{R}^n}}^{n-2}\xi_{\mathbb{R}^n}.
}
Hence $\xi_{\mathbb{R}^n}$ solves \eqref{EL_another}. Finally, by conformal invariance of \eqref{EL_another} and using \eqref{conformal_change}, we complete the proof for $\xi$.

Note that by definition of $\Sigma$ we have
\eq{
    \<\Sigma x, e_1\> = \<\Sigma x, x\> = 0.
}
Therefore, after replacing $\Sigma$ by $-\Sigma$, one can see that the same conclusion holds for $\bar{\xi}$.
\end{proof}

We recall the spectrum of the Hodge Laplacian on $\S^n$, which was first given in \cite{IK79}.

\begin{lemma}[cf. \cite{IK79}]
On $\S^n$, the spectrum of the Hodge Laplacian operator acting on $k$-forms consists of two series
\eq{
    \lambda_l^{(1)} = (k+l)(n-k+l+1), \quad \lambda_l^{(2)} = (k+l+1)(n-k+l), \quad l=0,1,2,\cdots,
}
where $\lambda_l^{(1)}$ corresponds to closed eigenforms and $\lambda_l^{(2)}$ corresponds to co-closed eigenforms.
\end{lemma}

We prove that $\xi$ is in fact the first co-closed eigenform of the Hodge Laplacian.

\begin{lemma}\label{lem6.3}
$\xi$ is a Killing 1-form and $\Delta\xi=2(n-1)\xi$. The same holds for $\bar{\xi}$.
\end{lemma}

\begin{proof}
By conformal invariance, \eqref{EL_another} also holds on $\S^n$ for $\xi$. Here $\mu_3 = 2^{\frac{3n}{4}}n^{-\frac{n}{2}}(n-1)^{\frac{n}{4}}$. Together with Lemma~\ref{lem6.1} we obtain
\eq{
    \curl^2\xi = 2(n-1)\xi.
}
In particular, $\xi$ is co-closed, and hence
\eq{
    \Delta\xi = \curl^2\xi = 2(n-1)\xi.
}
For any co-closed $1$-form $\alpha$ on $\S^n$, \cite[Corollary~1.1.4]{Semmelmann03} shows that $\alpha$ is a Killing $1$-form if and only if $\Delta\alpha = 2(n-1)\alpha$. Therefore $\xi$ is a Killing $1$-form. The argument for $\bar\xi$ is identical.
\end{proof}

\begin{remark}
The $1$-forms $\xi$ and $\bar\xi$ defined above form a special subclass of Killing $1$-forms, namely those of constant length. When $n=3$, they span the space of Killing $1$-forms. When $n>3$, there are many other Killing $1$-forms: the space of Killing $1$-forms (equivalently, Killing vector fields) has dimension $\frac{n(n+1)}{2}$, whereas the span of $\xi$ and $\bar\xi$ has dimension only $2n$.
\end{remark}

\begin{lemma}\label{lem6.4}
For any Killing 1-form $\xi$ and any $X\in\Gamma(T\S^n)$ we have $\nabla_X\rd\xi = -2X^\flat\wedge\xi$.
\end{lemma}

\begin{proof}
It suffices to prove it locally. Given any point $x_0\in\S^n$, let $\{e_i\}_{i=1}^n$ be a local orthonormal basis at $x_0$ with $\nabla e_i(x_0)=0$. Then we have
\eq{\label{eq_lemm6.4}
    \nabla_{e_i}\rd\xi - \rd\nabla_{e_i}\xi &= \nabla_{e_i}(e_j^\flat\w\nabla_{e_j}\xi) - e_j^\flat\w\nabla_{e_j}\nabla_{e_i}\xi = e_j^\flat\w (\nabla_{e_i}\nabla_{e_j}\xi - \nabla_{e_j}\nabla_{e_i}\xi)\\
    &= e_j^\flat\w R(e_i,e_j)\xi = e_j^\flat\w (\xi_je_i^\flat - \xi_i e_j^\flat) = - e_i^\flat\w\xi,
}
where we used the Riemann curvature $R(X,Y)Z=\<Y,Z\>X-\<X,Z\>Y$ on $\S^n$.
Since $\xi$ is a Killing 1-form, by definition we have
\eq{
    \nabla_{e_i}\xi = \frac{1}{2} e_i\ip\rd\xi.
}
It follows
\eq{
    \rd\nabla_{e_i}\xi &= \frac{1}{2}\rd(e_i\ip\rd\xi) = \frac{1}{2}e_j^\flat\w\nabla_{e_j}(e_i\ip\rd\xi)\\
    &= \frac{1}{2}e_j^\flat\w e_i\ip \nabla_{e_j}\rd\xi = \frac{1}{2}\nabla_{e_i}\rd\xi - \frac{1}{2}e_i\ip e_j^\flat\w\nabla_{e_j}\rd\xi\\
    &= \frac{1}{2}\nabla_{e_i}\rd\xi - \frac{1}{2}e_i\ip \rd^2\xi = \frac{1}{2}\nabla_{e_i}\rd\xi,
}
where we have used Proposition~\ref{basic_form}(7) in the fifth equality. Together with \eqref{eq_lemm6.4} we complete the proof.
\end{proof}

\begin{lemma}\label{lem6.5}
For any two Killing 1-forms $\xi,\eta$ on $\S^n$, we have
\begin{enumerate}
    \item $\<\curl\xi, \curl\eta\> + 4\<\xi,\eta\> \equiv{\rm const}$;
    \item $\<\curl\xi, \curl\eta\> = 2\<\nabla\xi,\nabla\eta\>$.
\end{enumerate}
\end{lemma}

\begin{proof}
(1) Using Lemma~\ref{lem6.4} we have for any vector field $X$
\eq{
    \nabla_X \<\curl\xi, \curl\eta\> &= \nabla_X \<\rd\xi, \rd\eta\> = \<\nabla_X\rd\xi,\rd\eta\> + \<\rd\xi, \nabla_X\rd\eta\> = -2\<X^\flat\w\xi,\rd\eta\> -2\<\rd\xi,X^\flat\w\eta\>.
}
By definition of Killing 1-forms we have
\eq{
    \nabla_X\<\xi,\eta\> = \<\nabla_X\xi,\eta\> + \<\xi,\nabla_X\eta\> = \frac{1}{2}\<X\ip\rd\xi,\eta\> + \frac{1}{2}\<\xi,X\ip\rd\eta\> = \frac{1}{2}\<\rd\xi,X^\flat\w\eta\> + \frac{1}{2}\<X^\flat\w\xi,\rd\eta\>,
}
where we used Proposition~\ref{basic_form}(2) in the last equality. Hence
\eq{
    \nabla_X\big(\<\curl\xi,\curl\eta\> + 4\<\xi,\eta\>\big) = 0, \quad \forall X\in\Gamma(T\S^n)
}
and the conclusion follows.
    
(2) By definition of Killing 1-forms we have
\eq{
    \<\nabla\xi,\nabla\eta\> = \<\nabla_{e_i}\xi,\nabla_{e_i}\eta\> = \frac{1}{4}\<e_i\ip\rd\xi,e_i\ip\rd\eta\> = \frac{1}{4}\<\rd\xi, e_i^\flat\w e_i\ip\rd\eta\> = \frac{1}{2}\<\rd\xi,\rd\eta\>,
}
where we used \eqref{direct_dual} in the third equality and Proposition~\ref{basic_form}(7) in the last. Hence we complete the proof.
\end{proof}

Using a similar method one can extend Lemma~\ref{lem6.4} and Lemma~\ref{lem6.5} to Killing $k$-forms. We state the result, though not necessarily for our purpose, and omit the proof.

\begin{lemma}
For any two Killing $k$-forms $\xi,\eta$ on $\S^n$, we have
\begin{enumerate}
    \item $\nabla_X\rd\xi = -(k+1)X^\flat\w\xi$ for any $X\in\Gamma(T\S^n)$;
    \item $\<\curl\xi,\curl\eta\> + (k+1)^2\<\xi,\eta\> \equiv{\rm const}$;
    \item $\<\curl\xi,\curl\eta\> = (k+1)\<\nabla\xi,\nabla\eta\>$. 
\end{enumerate}
\end{lemma}

\begin{remark}\label{rmk6.7}
The construction \eqref{construction_xi} of $\xi_{\mathbb{R}^n}$ does not rely on the choice of the unit vector $e_1$. In fact, one can replace $e_1$ by any other unit constant vector in $\mathbb{R}^n$ and obtain the same conclusion, just keeping in mind that $\Sigma$ should be replaced as well. For example, if we replace $e_1$ by $-e_2$ and change the sign of corresponding $\Sigma$, then the corresponding Killing 1-form is
\eq{
    \eta = \left[n\left( \frac 1 {1+|x|^2} \right)^2( -(1-|x|^2) e_2 -2 x_2 x +2\Sigma' x)\right]^\flat,
}
where

\eq{
    \Sigma' \coloneqq \left ( \begin{matrix}
    0 & &  & & \cdots &  1 \\
    & 0 &  &  &  &  &  \\
    & & 0 & -1 &  &  &  \\
    & & 1 & 0  &  &  &  \\
    \vdots & & & & \ddots &  &  \\
    -1 & & & &  &   0
    \end{matrix}\right).
}
\end{remark}

Let $\eta$ be defined as in Remark~\ref{rmk6.7}. Using the symmetry of $\S^n$, it is elementary to check that
\eq{\label{before_lemm6.6}
\int \<\xi,\eta\> = 0 \quad \hbox{and} \quad \<\xi,\eta\>\not\equiv0.
}

The following lemma can also be proved by direct computation on $\mathbb{R}^n$, but we give a short proof instead.
\begin{lemma}\label{lem6.6}
By our choice of $\eta$, we have $\int \<\curl\xi,\curl\eta\> = 0$.
\end{lemma}

\begin{proof}
For any first eigenform $\alpha\in\Omega^1$ of the Hodge Laplacian operator on $\S^n$, by the Weitzenb\"ock formula \eqref{sphere_Weitzenbock} we have
\eq{
    2(n-1)\alpha = \Delta\alpha = \nabla^*\nabla\alpha + (n-1)\alpha.
}
Therefore $\nabla^*\nabla\alpha = (n-1)\alpha$ and hence
\eq{
    \int \abs{\nabla\alpha}^2 = (n-1) \int \abs{\alpha}^2.
}
Choosing $\alpha = \xi+\eta$, $\alpha=\xi$, and $\alpha=\eta$, respectively, we have
\eq{
    \int \abs{\nabla\xi}^2 + 2\int\<\nabla\xi,\nabla\eta\> + \int\abs{\nabla\eta}^2 = (n-1)\int\abs{\xi}^2 + 2(n-1)\int\<\xi,\eta\> + (n-1)\int\abs{\eta}^2
}
and
\eq{
    \int\abs{\nabla\xi}^2 = (n-1)\int\abs{\xi}^2, \quad \int\abs{\nabla\eta}^2 = (n-1)\int\abs{\eta}^2.
}
In view of \eqref{before_lemm6.6}, we have
\eq{
    \int \<\nabla\xi,\nabla\eta\> = 0.
}
Together with Lemma~\ref{lem6.5}(3), the conclusion follows.
\end{proof}

\begin{corollary}\label{coro6.7}
By our choice of $\eta$, we have $\<\curl\xi,\curl\eta\> + 4\<\xi,\eta\> \equiv 0$.
\end{corollary}

\begin{proof}
It follows from \eqref{before_lemm6.6}, Lemma~\ref{lem6.5}(2), and Lemma~\ref{lem6.6}.
\end{proof}

Now we prove that $\xi$ is not a minimizer of $J_{\frac{n}{2},1}$. It suffices to prove that the second variation of $J_{\frac{n}{2},1}$ at $\xi$ in fact admits negative directions.

\begin{proposition}\label{prop6.8}
The second variation of $J_{\frac{n}{2},1}$ at $\xi$ admits negative directions.
\end{proposition}

\begin{proof}
For short we denote $a \coloneqq n\big(\frac{n-1}{2}\big)^{1/2}$ and $b \coloneqq \frac{n}{2}$. Then Lemma~\ref{lem6.1} becomes
\eq{
    \abs{\curl\xi} \equiv a, \quad \abs{\xi} \equiv b.
}
In view of Lemma~\ref{lem6.3}, we are able to compute the second variation at $\xi$:
\eq{
    \frac{\rd^2}{\rd t^2}\Big|_{t=0}J_{\frac{n}{2},1}(\xi+t\varphi) = &
    b^{-2n}
    \Bigg\{ \frac{n}{2}a^{n-4}b^{n} \omega_n^{-1} \Big( 2(n-1)\int \<\xi, \varphi\> \Big)^2
    + \Big(\frac{n}{2}-2\Big) a^{n-4}b^{n} \int \<\curl\xi,\curl\varphi\>^2\\
    &+ a^{n-2}b^{n} \int \abs{\curl\varphi}^2 - (n-2)a^{n}b^{n-4} \int \<\xi,\varphi\>^2
    - a^{n}b^{n-2} \int \abs{\varphi}^2 \Bigg\}.
}
Choose $\varphi=\eta$. By Corollary~\ref{coro6.7} we have
\eq{
    \int \<\curl\xi,\curl\eta\>^2 = 16 \int\<\xi,\eta\>^2.
}
Hence
\eq{
    \frac{\rd^2}{\rd t^2}\Big|_{t=0}J_{\frac{n}{2},1}(\xi+t\varphi)
    &= b^{-2n}
    \Bigg\{ 16\Big(\frac{n}{2}-2\Big) a^{n-4}b^{n} \int \<\xi,\eta\>^2 + a^{n}b^{n}\omega_n - (n-2)a^{n}b^{n-4} \int \<\xi,\eta\>^2
    - a^{n}b^{n}\omega_n \Bigg\}\\
    &= a^{n-2}b^{-n-2}\Big( \frac{b^2}{a^2}\cdot 8(n-4) - \frac{a^2}{b^2}(n-2) \Big)\int \<\xi,\eta\>^2\\
    &= 2a^{n-2}b^{-n-2}\Big( \frac{2(n-4)}{n-1} - (n-1)(n-2) \Big)\int \<\xi,\eta\>^2 \leq 0.
}
Recall that $\<\xi,\eta\>\not\equiv0$, hence the inequality is strict.
\end{proof}
\medskip

\noindent{\it Proof of Theorem~\ref{thm1.3}.}
Proposition~\ref{prop6.8} provides a direction along which the second variation of $J_{\frac{n}{2},1}$ at $\xi$ is strictly negative. Therefore $\xi$ cannot be a local minimizer, and Theorem~\ref{thm1.3} follows. \qed

\subsection{Proof of Theorem~\ref{thm1.2}}

Assume $n\equiv 3\pmod 4$ and set $k=\frac{n-1}{2}$. We prove that Killing $k$-forms are not local minimizers.

It is enough to show that for any $1<p<n$, the (formal) second variation of $J_{p,k}$ at any $\xi\in E_1$ admits a negative direction. The case of negative Killing $k$-forms is identical.

\begin{proposition}\label{prop6.9}
The (formal) second variation of $J_{p,\frac{n-1}{2}}$ at $\xi\in E_1$ admits negative directions.
\end{proposition}

\begin{proof}
For $\xi\in E_1$ the (formal) second variation is
\eq{
    \frac{\rd^2}{\rd t^2}\Big|_{t=0}J_{p,\frac{n-1}{2}}(\xi+t\varphi)
    = &\,q\Big(\frac{n+1}{2}\Big)^{q-2}\omega_n^{\frac{q}{p}-2}
    \Bigg\{ (q-p)\Big(\frac{n+1}{2}\Big)^2\omega_n^{-1} \Big( \int \<\xi, \varphi\> \Big)^2 + \int \abs{\curl \varphi}^2\\
    &\qquad + (p-2) \int \<\xi, \curl \varphi\>^2 - (q-2)\Big(\frac{n+1}{2}\Big)^2 \int \<\xi,\varphi\>^2 - \Big(\frac{n+1}{2}\Big)^2 \int \abs{\varphi}^2 \Bigg\}.
}

Choosing $\varphi = \varphi_{-1}\in E_{-1}$ yields
\eq{\label{second_variation_L_pq}
    \frac{\rd^2}{\rd t^2}\Big|_{t=0}J_{p,\frac{n-1}{2}}(\xi+t\varphi)
    = q\Big(\frac{n+1}{2}\Big)^{q-2}\omega_n^{\frac{q}{p}-2}
    \Bigg\{ (p-q)\Big(\frac{n+1}{2}\Big)^2 \int \<\xi,\varphi_{-1}\>^2 \Bigg\}.
}

By Proposition~\ref{prop5.1}, $\varphi_{-1}$ can be chosen so that $\<\xi,\varphi_{-1}\>\not\equiv 0$. Since $p-q<0$, the right-hand side of \eqref{second_variation_L_pq} is strictly negative, proving the claim.
\end{proof}

\begin{proof}[Proof of Theorem~\ref{thm1.2}]
It follows immediately from Proposition~\ref{prop6.9}.  
\end{proof}

\medskip

We now end the paper by an open question, which we believe is true. Define an action of $g\in\{1\}\times {\rm SO}(n-1)$ on vector fields by
\eq{u_g(x) \coloneqq g\cdot u(g^{-1}\cdot x).}
The Loss--Yau solutions constructed in \eqref{construction_xi} are invariant under this action. Are they minimizers of $J_2$ within the class of invariant vector fields? By contrast, our counterexample appears to be non-invariant under this group action. A similar question was also raised in \cite{MS25}.

\appendix

\section{Classification of eigenforms}\label{secA}

We now give a complete classification of curl eigenforms on $\S^n$. We believe this result is of independent interest.

\begin{theorem}\label{thmA.1}
Let $E_{\pm k}$ be as defined above. Then $E_1$ is the space of positive Killing $\frac{n-1}{2}$-forms. Moreover, for every $k\ge 1$ we have
\eq{
    E_{k} &= \Big\{ \alpha = (n+k-2)\sum_i f_i\xi_i - \sum_i \nabla f_i \ip*\xi_i \,\Big|\, \xi_i\in E_1,\ f_i\in P_{k-1},\ \alpha\ \text{co-closed} \Big\} \\
    &= \mathrm{span}\big\{ (n+k-2)f\xi - \nabla f\ip*\xi \,\big|\, \xi\in E_1,\ f\in P_{k-1} \big\} \cap \rd^*\Omega^{\frac{n-3}{2}},
}
and
\eq{
    E_{-k} &= \Big\{ \alpha = -(k+1)\sum_i f_i\xi_i - \sum_i \nabla f_i \ip*\xi_i \,\Big|\, \xi_i\in E_1,\ f_i\in P_{k+1},\ \alpha\ \text{co-closed} \Big\}\\
    &= \mathrm{span}\big\{ -(k+1)f\xi - \nabla f\ip*\xi \,\big|\, \xi\in E_1,\ f\in P_{k+1} \big\} \cap \rd^*\Omega^{\frac{n-3}{2}}.
}
\end{theorem}

\begin{proof}
The case $k=1$ is clear, so we only classify $E_k$ for $k\ge 2$ and $E_{-k}$ for $k\ge 1$.

Let $\alpha$ be a co-closed $\frac{n-1}{2}$-form. By Proposition~\ref{Killing_properties}(3) we may write
\eq{
    \alpha = \sum_i f_i\xi_i \quad \text{with}\quad \xi_i\in E_1,
}
for some functions $f_i$. By Proposition~\ref{Killing_properties}(4),
\eq{\label{classification_eq0}
    \curl \alpha = \curl\Big(\sum_i f_i\xi_i\Big) = \frac{n+1}{2} \sum_i f_i\xi_i - \sum_i \nabla f_i\ip *\xi_i.
}
Consequently,
\eq{\label{classification_eq1}
    \curl\curl\alpha &= \frac{n+1}{2} \curl\Big(\sum_i f_i\xi_i\Big) - \curl\Big(\sum_i \nabla f_i\ip *\xi_i\Big)\\
    &= \frac{(n+1)^2}{4} \sum_i f_i\xi_i - \frac{n+1}{2} \sum_i \nabla f_i\ip *\xi_i - \curl\Big(\sum_i \nabla f_i\ip *\xi_i\Big).
}
Since $\alpha$ is co-closed, $\curl\curl\alpha=\Delta\alpha$.

Next, for $\xi\in E_1$ and a function $f$, the Weitzenb\"ock formula \eqref{sphere_Weitzenbock} gives
\begin{align}
    \Delta(f\xi) &= \nabla^*\nabla(f\xi) + \frac{(n+1)(n-1)}{4}f\xi\\
    &= (\nabla^*\nabla f)\xi + f\nabla^*\nabla\xi - 2\nabla_{e_j}f\,\nabla_{e_j}\xi + \frac{(n+1)(n-1)}{4}f\xi\\
    &= (\nabla^*\nabla f)\xi + f\Big(\Delta\xi - \frac{(n+1)(n-1)}{4}\xi\Big) - 2\nabla_{\nabla f}\xi + \frac{(n+1)(n-1)}{4}f\xi\\
    &= (\nabla^*\nabla f)\xi + \frac{(n+1)^2}{4}f\xi - 2\nabla f \ip*\xi,
\end{align}
where we used Proposition~\ref{Killing_properties}(1) in the last equality. Therefore
\eq{ \label{classification_eq2}
    \curl\curl\alpha = \Delta\alpha = \sum_i (\nabla^*\nabla f_i)\xi_i + \frac{(n+1)^2}{4} \sum_i f_i\xi_i - 2\sum_i \nabla f_i \ip*\xi_i.
}
Comparing \eqref{classification_eq1} and \eqref{classification_eq2} yields
\eq{\label{classification_eq3}
    \curl \Big(\sum_i \nabla f_i\ip *\xi_i\Big) = \sum_i (-\nabla^*\nabla f_i)\xi_i - \frac{n-3}{2} \sum_i \nabla f_i \ip*\xi_i.
}
In particular, if $f\in P_k$, then
\eq{\label{classification_eq4}
    \curl \Big(\sum_i \nabla f_i\ip *\xi_i\Big) = -k(n+k-1) \sum_i f_i\xi_i - \frac{n-3}{2} \sum_i \nabla f_i \ip*\xi_i.
}
Combining \eqref{classification_eq0} and \eqref{classification_eq4}, we obtain
\eq{
    \curl \Big( (n+k-1)\sum_i f_i\xi_i - \sum_i \nabla f_i\ip*\xi_i \Big)
    = \Big( \frac{n+1}{2} + k \Big) \Big( (n+k-1)\sum_i f_i\xi_i - \sum_i \nabla f_i\ip*\xi_i \Big)
}
and
\eq{\label{classification_eq4.5}
    \curl \Big( -k\sum_i f_i\xi_i - \sum_i \nabla f_i\ip*\xi_i \Big)
    = \Big( -\frac{n+1}{2} - (k-2) \Big) \Big( -k\sum_i f_i\xi_i - \sum_i \nabla f_i\ip*\xi_i \Big).
}

Define
\eq{
    V_k \coloneqq \Big\{ \alpha = (n+k-1)\sum_i f_i\xi_i - \sum_i \nabla f_i \ip*\xi_i \,\Big|\, \xi_i\in E_1,\ f_i\in P_{k},\ \alpha\ \text{co-closed} \Big\}
}
and
\eq{
    W_k \coloneqq \Big\{ \alpha = -k\sum_i f_i\xi_i - \sum_i \nabla f_i \ip*\xi_i \,\Big|\, \xi_i\in E_1,\ f_i\in P_{k},\ \alpha\ \text{co-closed} \Big\}.
}
For $k\ge 1$ (and with the convention $E_0=\{0\}$) we have
\eq{\label{classification_eq5}
    V_{k} \subset E_{k+1},\quad W_{k} \subset E_{-(k-1)}.
}

Conversely, let $\varphi_{k+1}\in E_{k+1}$. Writing
\eq{
    \varphi_{k+1} = \sum_i h_i\xi_i
}
for some functions $h_i$, Proposition~\ref{new_eigenfunction} implies $h_i\in P_k$. Since $\varphi_{k+1}$ is co-closed, we obtain
\eq{
    \varphi_{k+1} \in \Big\{ \alpha = \sum_i f_i\xi_i \,\Big|\, \xi_i\in E_1,\ f_i\in P_k,\ \alpha\ \text{co-closed} \Big\} \subset V_k\oplus W_k.
}
Hence $E_{k+1}\subset V_k\oplus W_k$. The same argument gives $E_{-(k-1)}\subset V_k\oplus W_k$. Together with \eqref{classification_eq5},
\eq{
    V_k \oplus W_k \subset E_{k+1}\oplus E_{-(k-1)} \subset V_k \oplus W_k,
}
so $E_{k+1}=V_k$ and $E_{-(k-1)}=W_k$, as claimed.
\end{proof}

\begin{remark}
(1) For any function $f$ and any form $\alpha$,
\eq{
    \rd^*(f\alpha) = f\rd^*\!\alpha - \nabla f \ip\alpha.
}
In particular, for $\xi\in E_1$ we have
\eq{
    \rd^*(f\xi) = -\nabla f\ip\xi.
}
Moreover, using Proposition~\ref{Killing_properties}(4),
\eq{
    \rd^*(\nabla f \ip*\xi)
    &= \frac{n+1}{2}\rd^*(f\xi) - \rd^*\Big(\frac{n+1}{2}f\xi - \nabla f \ip*\xi\Big)\\
    &= -\frac{n+1}{2}\nabla f\ip\xi - \rd^*(*\rd(f\xi))\\
    &= -\frac{n+1}{2}\nabla f\ip\xi.
}
Consequently, for $\alpha=\sum_i f_i\xi_i$ the co-closedness constraint is equivalent to
\eq{
    \sum_i \nabla f_i\ip\xi_i = 0.
}

(2) If $k=2$, then Theorem~\ref{thmA.1} reduces to our earlier classification Proposition~\ref{classification}. Indeed, the condition $\sum_i \nabla f_i\ip\xi_i = 0$ from (1) implies
\eq{
    \sum_i f_i\xi_i + \sum_i \nabla f_i\ip*\xi_i = 0 \qquad \text{for } f_i\in P_1.
}
Equivalently (see $k=1$ in \eqref{classification_eq4.5}, or (1) together with Proposition~\ref{Killing_properties}(4)),
\eq{
    \rd^* \Big(\sum_i f_i\xi_i + \sum_i \nabla f_i\ip*\xi_i\Big)
    = \rd \Big(\sum_i f_i\xi_i + \sum_i \nabla f_i\ip*\xi_i\Big) = 0.
}

\item There is an analogous classification using negative Killing forms as a trivialization; see Proposition~\ref{Killing_properties}(3).
\end{remark}

\section{A new conformal invariant}\label{secB}

On a compact Riemannian manifold $(M,g)$, consider the conformal class
\eq{[g]=\{\tilde g\mid \tilde g=u^2 g\}.}
Denote by $\lambda_1^+(\tilde g)$ the first positive eigenvalue of the curl operator with respect to $\tilde g\in[g]$, and define the conformal invariant
\eq{
    \mu([g]) \coloneqq \inf_{\tilde g\in[g]} \lambda_1^+(\tilde g)\,\rVol(\tilde g)^{\frac{1}{n}}.
}
It is known that $\mu([g])\in(0,\infty)$; see \cite{Jammes07}. Determining whether the infimum is achieved (and describing optimizers) appears to be difficult; this remains open even for $n=3$ and $(M,g)=(\S^3,g_{\rm st})$.  The invariant $\mu([g])$ is closely related to the functional $J_1$.

\begin{lemma}
\eq{
    \mu([g]) = \inf_{\int \< \curl\alpha, \alpha\> >0} J_1(\alpha,g).
}
\end{lemma}

\begin{proof}
The argument is standard; we include it for completeness.

Fix $\tilde g\in[g]$, and let $\lambda_1^+(\tilde g)$ and $\alpha_1$ be the first positive eigenvalue and a corresponding eigenform of the curl operator. Using the conformal invariance of $J_1$ and H\"older's inequality,
\eq{
    \inf J_1(\alpha,g)
    = \inf J_1(\alpha,\tilde g)
    \le J_1(\alpha_1,\tilde g)
    = \lambda_1^+(\tilde g)\,
    \frac{\big(\int \abs{\alpha_1}_{\tilde g}^{\frac{2n}{n+1}}\rdV_{\tilde g}\big)^{\frac{n+1}{n}}}{\int \abs{\alpha_1}_{\tilde g}^2\,\rdV_{\tilde g}}
    \le \lambda_1^+(\tilde g)\,\rVol(\tilde g)^{\frac{1}{n}}.
}
Taking the infimum over $\tilde g\in[g]$ gives $\inf J_1(\alpha,g)\le \mu([g])$.

For the reverse inequality, fix $\varepsilon>0$ and choose $\alpha_\varepsilon$ with $\int\<\curl\alpha_\varepsilon,\alpha_\varepsilon\> >0$ such that
\eq{
    J_1(\alpha_\varepsilon,g) \le \inf J_1(\alpha,g) + \varepsilon.
}
After scaling, we may assume
\eq{
    \int \abs{\curl_g\alpha_\varepsilon}_g^{\frac{2n}{n+1}}\rdV_g = 1.
}
Define $u\coloneqq \abs{\curl_g\alpha_\varepsilon}_g^{\frac{2}{n+1}}$ (assuming for the moment that $\curl\alpha_\varepsilon$ is nowhere vanishing) and set $\tilde g\coloneqq u^2 g$. Then
\eq{
    \rVol(\tilde g)=\int u^n\rdV_g = \int \abs{\curl_g\alpha_\varepsilon}_g^{\frac{2n}{n+1}}\rdV_g = 1.
}
Using the conformal transformation formulas from Subsection~\ref{sec2.2},
\eq{
    \Big(\int \abs{\curl_{\tilde g}\alpha_\varepsilon}_{\tilde g}^{\frac{2n}{n+1}}\rdV_{\tilde g}\Big)^{\frac{n+1}{n}}
    = 1 = \int \abs{\curl_{\tilde g}\alpha_\varepsilon}_{\tilde g}^2\,\rdV_{\tilde g}.
}
Therefore,
\eq{
    \inf J_1(\alpha,g) + \varepsilon
    \ge J_1(\alpha_\varepsilon,g)
    = J_1(\alpha_\varepsilon,\tilde g)
    = \frac{\int \abs{\curl_{\tilde g}\alpha_\varepsilon}_{\tilde g}^2\,\rdV_{\tilde g}}{\int \<\curl_{\tilde g}\alpha_\varepsilon,\alpha_\varepsilon\>_{\tilde g}\,\rdV_{\tilde g}}
    \ge \lambda_1^+(\tilde g)
    = \lambda_1^+(\tilde g)\,\rVol(\tilde g)^{\frac{1}{n}}
    \ge \mu([g]).
}
Letting $\varepsilon\to 0$ yields $\inf J_1(\alpha,g)\ge \mu([g])$.

If $u$ has zeros, one can instead take $u_\varepsilon\coloneqq \sqrt{u^2+\varepsilon^2}$ and run the same argument.
\end{proof}

A conformal metric $\tilde g\in[g]$ achieving $\mu([g])$ is called an \emph{optimizer}. There are interesting results on closed $3$-manifolds \cite{EGP25} and on $3$-manifolds with boundary (where one also needs boundary conditions); see \cites{EP20,EGP22}. However, for $\S^3$ this problem remains open.

A natural conjecture, mentioned in the introduction,  is that the standard round metric $g_{\rm st}$ on $\S^n$ is an optimizer, which is equivalent to 
\eq{\label{best_constant}
  S_1=  \inf_{\int \< \curl\alpha, \alpha\> >0} J_1(\alpha,g_{\rm st})
    = \frac{n+1}{2}\,\omega_n^{\frac{1}{n}}.
}

\section{A lower bound in dimension 3}\label{secC}

In this appendix we prove a quantitative lower bound for $S_1$ in dimension $n=3$. Recall that
\eq{
    S_1 \coloneqq \inf\Big\{ \frac{ \Big(\int_{\S^3} \abs{\curl\alpha}^{\frac{3}{2}} \,\rdV \Big)^{\frac{4}{3}} }{ \int_{\S^3} \< \curl\alpha, \alpha\> \,\rdV } \,\Big|\, \rd^*\!\alpha=0,\ \int_{\S^3} \< \curl\alpha, \alpha\> \,\rdV > 0 \Big\}.
}
By conformal invariance, this is equivalent to the corresponding variational problem on $\R^3$.
The exact value of $S_1$ is still unknown. If the conjecture that Killing $1$-forms (equivalently, Killing vector fields) are global minimizers is true, then
\eq{
    S_1 = 2\omega_3^{\frac{1}{3}} \approx 5.405.
}
In \cite[Appendix~A]{EGP25} the authors derived a rough lower bound for $S_1$ by applying Cauchy--Schwarz to reduce to a scalar inequality and then invoking the sharp Sobolev constant for functions~\cite{Lieb83}:
\eq{
    S_1 \ge \Big( \frac{16}{\pi} \Big)^{\frac{1}{3}} \approx 1.721.
}

We improve this bound by relating the curl quotient to the sharp spinorial Sobolev inequality
\eq{\label{spin_Sobolev}
    \frac{ \Big(\int_{\S^3} \abs{\D\varphi}^{\frac{3}{2}} \,\rdV \Big)^{\frac{4}{3}} }{ \int_{\S^3} \< \D\varphi, \varphi\> \,\rdV } \ge \frac{3}{2}\omega_3^{\frac{1}{3}},
    \qquad \text{whenever } \int_{\S^3} \< \D\varphi, \varphi\> \,\rdV > 0,
}
where $\D$ denotes the Dirac operator and $\varphi$ is a spinor field. Equality holds if and only if $\varphi$ is a Killing spinor, up to composition with a conformal transformation; see~\cite{A03}.

\begin{theorem}\label{thmC.1}
$S_1 \ge \frac{3}{2}\omega_3^{1/3} \approx 4.054$.
\end{theorem}

The proof starts from the following observation. Let $\xi$ be a unit-length Killing spinor on $\S^3$, satisfying
\eq{
    \nabla_X\xi = -\frac{1}{2}X\cdot\xi, \quad\forall X\in\Gamma(T\S^3).
}
In particular,
\eq{
    \D\xi = \sum_{i=1}^3 e_i\cdot\nabla_{e_i}\xi = \frac{3}{2}\xi.
}
Using the elementary formula for any $1$-form $\alpha$ and any spinor field $\varphi$
\eq{
    \D(\alpha\cdot\varphi) = -\alpha\cdot\D\varphi + (\rd+\rd^*)\alpha\cdot\varphi - 2\nabla_{\alpha^\sharp}\varphi,
}
we obtain that for any co-closed $1$-form $\alpha$,
\eq{
    \D(\alpha\cdot\xi) = -\frac{1}{2}\alpha\cdot\xi + \rd\alpha\cdot\xi.
}
Moreover, in dimension $3$ we have $\rd\alpha\cdot\xi = *\rd\alpha\cdot\xi$ (this is a standard feature of the irreducible real spin representation; see, for instance, \cite[Chapter~1, Proposition~5.9]{Lawson} or \cite[Proposition~1.30]{Bourguignon2015}).
Consequently,
\eq{
    \curl\alpha\cdot\xi = (\D+\frac{1}{2})(\alpha\cdot\xi).
}
Taking length yields
\eq{
    \abs{\curl\alpha} = \abs{(\D+\frac{1}{2})(\alpha\cdot\xi)},
}
and moreover
\eq{
    \<\curl\alpha,\alpha\> = \<\curl\alpha\cdot\xi,\alpha\cdot\xi\> = \<(\D+\frac{1}{2})(\alpha\cdot\xi),\alpha\cdot\xi\>.
}
Therefore
\eq{\label{eq_key_relation}
    \frac{ \Big(\int_{\S^3} \abs{\curl\alpha}^{\frac{3}{2}} \,\rdV \Big)^{\frac{4}{3}} }{ \int_{\S^3} \< \curl\alpha, \alpha\> \,\rdV } = \frac{ \Big(\int_{\S^3} \abs{(\D+\frac{1}{2})(\alpha\cdot\xi)}^{\frac{3}{2}} \,\rdV \Big)^{\frac{4}{3}} }{ \int_{\S^3} \< (\D+\frac{1}{2})(\alpha\cdot\xi), \alpha\cdot\xi\> \,\rdV }.
}
Thus the original problem is reduced to a Sobolev-type quotient for the shifted Dirac operator $\D+\frac{1}{2}$. While $\D$ is conformally covariant, the shift breaks conformal covariance; nevertheless, we can still compute the corresponding sharp constant.

\begin{lemma}\label{lemC.2}
On $\S^3$ there holds
\eq{
    \inf \Big\{ \frac{ \Big(\int \abs{(\D+\frac{1}{2})\varphi}^{\frac{3}{2}} \,\rdV \Big)^{\frac{4}{3}} }{ \int \< (\D+\frac{1}{2})\varphi, \varphi\> \,\rdV } \,\Big|\, \int \< (\D+\frac{1}{2})\varphi, \varphi\> \,\rdV > 0 \Big\} = \frac{3}{2}\omega_3^{\frac{1}{3}}.
}
\end{lemma}
\begin{proof}
Let us denote the infimum by $A$. Note that on $\S^3$,
\eq{
    {\rm spec}(\D) = \{ \pm(\frac{3}{2}+k) \,|\, k\ge0 \}. 
}
Since $\D+\frac{1}{2}$ is invertible, we may rewrite
\eq{
    A = \inf \Big\{ \frac{ \Big(\int \abs{\varphi}^{\frac{3}{2}} \,\rdV \Big)^{\frac{4}{3}} }{ \int \< (\D+\frac{1}{2})^{-1}\varphi, \varphi\> \,\rdV } \,\Big|\, \int \< (\D+\frac{1}{2})^{-1}\varphi, \varphi\> \,\rdV > 0 \Big\}.
}
One checks that
\eq{
    \D^{-1} - (\D+\frac{1}{2})^{-1} > 0,
}
which yields the strict comparison
\eq{\label{ineq_strict}
    \int \< (\D+\frac{1}{2})^{-1}\varphi, \varphi\> \,\rdV < \int \< \D^{-1}\varphi, \varphi\> \,\rdV.
}
Consequently,
\eq{
    A \ge \frac{ \Big(\int \abs{\varphi}^{\frac{3}{2}} \,\rdV \Big)^{\frac{4}{3}} }{ \int \< \D^{-1}\varphi, \varphi\> \,\rdV } \ge \frac{3}{2}\omega_3^{\frac{1}{3}},
}
where the last inequality is equivalent to \eqref{spin_Sobolev}.

To obtain the matching upper bound, we construct a minimizing sequence via degenerating conformal transformations. Recall \eqref{eq_conformal_transformation}, and choose parameters $b_\epsilon\to N$ so that the conformal maps $\Xi_\epsilon\in{\rm Conf}(\S^3,g)$ degenerate to the point $N$ as $\epsilon\to0$. Recall the Killing spinor $\xi$ satisfies
\eq{\label{eq_conformal_Killing}
    \D\xi=\frac{3}{2}\xi, \quad \abs{\xi}=1.
}
Let $\tilde{g}=\Xi_\epsilon^*g=u_\epsilon^2g$ and $\tilde{\xi}\coloneqq\Xi_\epsilon^*\xi$. Then
\eq{
    \D_{\tilde{g}}\tilde{\xi}=\frac{3}{2}\tilde{\xi}, \quad \abs{\tilde{\xi}}_{\tilde{g}}=1.
}
View $\varphi_\epsilon=u_\epsilon\tilde{\xi}$ as a spinor associated to $g$ via the isometry between the spinor bundles $\Sigma_{\tilde{g}}$ and $\Sigma_g$. The conformal covariance of $\D$ implies
\eq{
    \D_g\varphi_\epsilon = u_\epsilon^2\D_{\tilde{g}}\tilde{\xi} = \frac{3}{2}u_\epsilon\varphi_\epsilon, \quad \abs{\varphi_\epsilon}_g=u_\epsilon.
}
It follows
\eq{
    \int \abs{\varphi_\epsilon}^3 \rdV_g = \int u_\epsilon^3 \rdV_g = \int \rdV_{\Xi_\epsilon^*g} = \int \Xi_\epsilon^*\rdV_g = \int \rdV_g = \omega_3.
}
Since $\Xi_\epsilon$ degenerates to a point as $\epsilon\to0$, $u_\epsilon$ uniformly converges to $0$ on each compact set of $\S^3\backslash\{N\}$. We claim that
\eq{
    \norm{\varphi_\epsilon}_p \to 0 \quad\hbox{as}\quad \epsilon\to0 \quad\hbox{for any}\quad p<3.
}
In fact, consider $\S^3 = (\S^3\backslash B_{\delta}(N))\cup B_{\delta}(N)$. The uniform convergence implies $\norm{\varphi_\epsilon}_{L^p(\S^3\backslash B_{\delta}(N))} \to 0$, while H\"older's inequality implies
\eq{
    \int_{B_\delta(N)}\abs{\varphi_\epsilon}^p\, \rdV_g = \int_{B_\delta(N)}u_\epsilon^p\, \rdV_g \le \left(\int_{\S^3}u_\epsilon^3\,\rdV_g\right)^{\frac{p}{3}} \rVol(B_\delta(N))^{1-\frac{p}{3}}.
}
Let $\delta$ be arbitrarily small and the claim follows.
The Minkowski inequality then gives
\eq{
    \Big| \norm{\D\varphi_\epsilon+\frac{1}{2}\varphi_\epsilon}_{\frac{3}{2}} - \norm{\D\varphi_\epsilon}_{\frac{3}{2}} \Big| \le \frac{1}{2}\norm{\varphi_\epsilon}_{\frac{3}{2}} \to 0 \quad\hbox{as}\ \epsilon\to0,
}
where
\eq{
    \norm{\D\varphi_\epsilon}_{\frac{3}{2}} = \frac{3}{2}\norm{\varphi_\epsilon}_3^2 = \frac{3}{2}\omega_3^{\frac{2}{3}}.
}
Hence
\eq{
    \norm{\D\varphi_\epsilon+\frac{1}{2}\varphi_\epsilon}_{\frac{3}{2}} \to \frac{3}{2}\omega_3^{\frac{2}{3}}\quad\hbox{as}\ \epsilon\to0. 
}
As a consequence,
\eq{
    \frac{ \Big(\int \abs{(\D+\frac{1}{2})\varphi_\epsilon}^{\frac{3}{2}} \,\rdV \Big)^{\frac{4}{3}} }{ \int \< (\D+\frac{1}{2})\varphi_\epsilon, \varphi_\epsilon\> \,\rdV } = \frac{ \norm{\D\varphi_\epsilon+\frac{1}{2}\varphi_\epsilon}_{\frac{3}{2}}^2 }{ \frac{3}{2}\omega_3 + \frac{1}{2}\norm{\varphi_\epsilon}_2^2 } \to \frac{3}{2}\omega_3^{\frac{1}{3}}.
}

\end{proof}

\begin{remark}
The infimum in Lemma~\ref{lemC.2} is not attained, because \eqref{ineq_strict} is strict for every admissible $\varphi$.

\noindent (2) The same argument shows that the conclusion remains valid if $\frac{1}{2}$ is replaced by any $a\in(0,\frac{3}{2})$. Moreover, the statement extends to higher dimensions.
\end{remark}

Finally, we prove Theorem~\ref{thmC.1}.

\medskip

\begin{proof}[Proof of Theorem~\ref{thmC.1}]
This is a consequence of \eqref{eq_key_relation} and Lemma~\ref{lemC.2}.
\end{proof}

Since the curl operator shares several features with the Dirac operator, it is natural to ask whether the argument above can be refined to yield an affirmative answer to this conjecture.

\medskip

\noindent\textsc{Acknowledgments.} G.W. would like to thank Professor Rupert Frank for helpful discussions.
M.Z. was supported by the China Scholarship Council (CSC), Grant No.~202306270117.

\bigskip

\noindent\textsc{Data availability.}
No new data were created or analyzed in this study; data sharing is not applicable.

\medskip

\noindent\textsc{Competing interests.}
The authors declare that they have no competing interests.

\printbibliography

\end{document}